\documentclass[12pt,reqno]{amsart}
\usepackage{amssymb,amsmath,amsthm,hyperref}
\newcommand{\sg}{\textnormal{sg}}

\newtheorem{theorem}{Theorem}
\newtheorem{lemma}[theorem]{Lemma}
\newtheorem{corollary}[theorem]{Corollary}
\newtheorem{proposition}[theorem]{Proposition}

\theoremstyle{remark}

\theoremstyle{definition}

\numberwithin{theorem}{section} 
\numberwithin{equation}{section}
\numberwithin{example}{section}

\newcommand{\Z}{\mathbb{Z}}

\renewcommand{\Im}{\text {\rm Im}}

\title[On Hecke-type double-sums and general string functions]{On Hecke-type double-sums and general string functions for the affine Lie algebra $A_{1}^{(1)}$}

\begin{document}

\date {6 October 2021}



\begin{abstract}
We demonstrate how formulas that express Hecke-type double-sums in terms of theta functions and Appell--Lerch functions---the building blocks of Ramanujan's mock theta functions---can be used to give general string function formulas for the affine Lie algebra $A_{1}^{(1)}$ for levels  $N=1,2,3,4$. 
\end{abstract}

\author{Eric T. Mortenson}
\address{Department of Mathematics and Computer Science, Saint Petersburg State University, Saint Petersburg, 199178, Russia}
\email{etmortenson@gmail.com}

\maketitle


\section{Notation}\label{section:notation}
 Let $q$ be a complex number where $q:=e^{2\pi i \tau}$ and $\tau\in\mathfrak{H}:=\{z\in \mathbb{C} | \Im{z}>0\}$.  Define $\mathbb{C}^*:=\mathbb{C}-\{0\}$.  We recall basic notation such as
\begin{gather*}
(x)_n=(x;q)_n:=\prod_{i=0}^{n-1}(1-q^ix), \ \ (x)_{\infty}=(x;q)_{\infty}:=\prod_{i= 0}^{\infty}(1-q^ix),\\
{\text{and }}\ \ j(x;q):=\sum_{n=-\infty}^{\infty}(-1)^nq^{\binom{n}{2}}x^n=(x)_{\infty}(q/x)_{\infty}(q)_{\infty},
\end{gather*}
where in the last line the equivalence of product and sum follows from Jacobi's triple product identity.  We draw the reader's attention to the fact that $j(q^n;q)=0$ for $n\in \mathbb{Z}.$

Let $a$ and $m$ be integers with $m$ positive.  We give special notation to frequently encountered theta functions
\begin{gather*}
J_{a,m}:=j(q^a;q^m), \ \ \overline{J}_{a,m}:=j(-q^a;q^m), \ J_m:=J_{m,3m}=\prod_{i= 1}^{\infty}(1-q^{mi}),
\end{gather*}
and we also recall Dedekind's eta-function:
\begin{equation*}
\eta(\tau):=q^{\tfrac{1}{24}}\prod_{n=1}^{\infty}(1-q^{n}).
\end{equation*}

We define an Appell-Lerch function as follows.   Let $x,z\in\mathbb{C}^*$ with neither $z$ nor $xz$ an integral power of $q$. Then
\begin{equation}
m(x,q,z):=\frac{1}{j(z;q)}\sum_r\frac{(-1)^rq^{\binom{r}{2}}z^r}{1-q^{r-1}xz}.\label{equation:mxqz-def}
\end{equation}

Lastly, we recall a useful form of a Hecke-type double-sum \cite{HM}.  Let $x,y\in\mathbb{C}^*$, then
\begin{equation}
f_{a,b,c}(x,y,q):=\Big ( \sum_{r,s\ge 0 }-\sum_{r,s<0}\Big )(-1)^{r+s}x^ry^sq^{a\binom{r}{2}+brs+c\binom{s}{2}},\label{equation:fabc-def}
\end{equation}
where we can also write
\begin{equation}
f_{a,b,c}(x,y,q)=\sum_{\substack{r,s\in\mathbb{Z}\\ \sg(r)=\sg(s)}}\sg(r)(-1)^{r+s}x^ry^sq^{a\binom{r}{2}+brs+c\binom{s}{2}}, \ \sg(r):=\begin{cases}
1 & \textup{if} \ r\ge0,\\
-1 & \textup{if} \  r<0.
\end{cases}
\end{equation}

\section{Introduction}\label{section:introduction}

In \cite{KP}, Kac and Peterson give several examples of string functions for affine Lie algebras of type $A_{1}^{(1)}$ that have beautiful evaluations in terms of theta functions.  See also \cite{KW1, KW2}.  Their string functions are closely related to the real quadratic fields $\mathbb{Q}(\sqrt{m(m+2)})$.  Indeed, if we fix a positive integer $m$, their string functions are of the form \cite[p. 260]{KP}:
\begin{equation}
\eta(\tau)^3c_{\lambda}^{\Lambda}(\tau)
=\sum_{\substack{(x,y)\in\mathbb{R}^2\\ -|x|<y\le |x|\\
(x,y) \ \textup{or} \ (1/2-x,1/2+y)\in ((N+1)/2(m+2),n/2m)+\mathbb{Z}^2}}
\sg(x)q^{(m+2)x^2-my^2},\label{equation:KP-fullGlory}
\end{equation}
where $N$ and $n$ are integers with $n\equiv N \pmod 2$.  Here we will change the notation from $c_{\lambda}^{\Lambda}(\tau)$ and write
\begin{equation}
c_{N-m,m}^{N-\ell,\ell}(\tau),
\end{equation}
where we have replaced Kac and Peterson's notation $(n,N,m)$ with $(m,\ell,N)$ of \cite{SW}.   In this form, $m$ and $\ell$ parametrize the maximal (resp. highest) weight in terms of the fundamental weights of the affine Kac--Moody algebra $\mathfrak{g}=A_1^{(1)}$.  See \cite{KP} and \cite{SW} for more details on string functions.

From \cite{SW} we recall that $m,\ell, N$ are integers with $N\ge1$, $\ell \in\{0,1,2,\dots,N \}$, and $m\equiv \ell \pmod 2$.   From \cite[p. 260]{KP}, \cite{SW} we find that
\begin{equation}
c_{N-m,m}^{N-\ell,\ell}:=C_{m,\ell}^{N}(q)
=q^{s(m,\ell,N)}\cdot \mathcal{C}_{m,\ell}^{N}(q),\label{equation:gen-fabc-SW}
\end{equation}
where
\begin{multline}\label{sffinal}
\mathcal{C}_{m,\ell}^N(q)=\frac{1}
{J_{1}^3}\sum_{j\in\Z}
\sum_{i\in\mathbb{N}} (-1)^i q^{\frac{1}{2}i(i+m)+j((N+2)j+\ell+1)} \\
\times
\Bigl\{q^{\frac{1}{2}i(2(N+2)j+\ell+1)}-q^{-\frac{1}{2}i(2(N+2)j+\ell+1)}\Bigr\},
\end{multline}
and
\begin{equation}
s(m,\ell,N):=-\frac{1}{8}+\frac{(\ell+1)^2}{4(N+2)}-\frac{m^2}{4N}.\label{equation:s-def}
\end{equation}
In \cite{MPS} we derived the useful form:
\begin{equation}
\mathcal{C}_{m,\ell}^{N}(q)=\frac{1}{J_{1}^3}\cdot f_{1,1+N,1}(q^{1+\tfrac{1}{2}(m+\ell)},q^{1-\tfrac{1}{2}(m-\ell)},q).
\label{equation:SW-fabc}
\end{equation}

Identity (\ref{equation:SW-fabc}) is very useful in computing the modularity of string functions.  In \cite{HM}, one finds many formulas where Hecke-type double-sums are expressed in terms of Appell-Lerch functions and theta functions.   Appell--Lerch functions are the building blocks of Ramanujan's mock theta functions.   A simple example of general results found in \cite{HM} reads
\begin{align}
f_{1,2,1}(x,y,q)&=j(y;q)m\big (\frac{q^2x}{y^2},q^3,-1\big )+j(x;q)m\big (\frac{q^2y}{x^2},q^3,-1\big )
\label{equation:f121}\\
&\ \ \ \ \ \ \  - \frac{yJ_3^3j(-x/y;q)j(q^2xy;q^3)}
{\overline{J}_{0,3}j(-qy^2/x;q^3)j(-qx^2/y;q^3)}.\notag
\end{align}
Such double-sum formulas provide a straightforward method for proving identities for Ramanujan's mock theta functions; in particular, the formulas give new proofs of the mock theta conjectures \cite{H1, H2, HM}.  As an example, one sees from (\ref{equation:f121}) that
\begin{align}
f_{1,2,1}(q,q,q)&=j(q;q)m(q,q^3,-1)+j(q;q)m(q,q^3,-1)
-\frac{1}{\overline{J}_{0,3}}\cdot \frac{qJ_{3}^3\overline{J}_{0,1}J_{4,3}}
{\overline{J}_{2,3}^2}\label{equation:f121-calculation}\\
&=0+0+J_1^2=J_1^2,\notag
\end{align}
where we remind the reader that $j(x;q)=0$ if and only if $x$ is an integral power of $q$.

String functions satisfy many symmetries \cite{SW}:
{\allowdisplaybreaks \begin{align}
C_{m,\ell}^{N}(q)&=C_{-m,\ell}^{N}(q),\label{equation:string-symmetry-1}\\
C_{m,\ell}^{N}(q)&=C_{2N-m,\ell}^{N}(q),\label{equation:string-symmetry-2}\\
C_{m,\ell}^{N}(q)&=C_{N-m,N-\ell}^{N}(q).\label{equation:string-symmetry-3}
\end{align}}%
Looking for a computationally easier approach to using (\ref{equation:SW-fabc}) and formulas found in \cite{HM} led to new symmetries.  The author, Postnova, and Solovyev found \cite{MPS}:

\begin{theorem} \cite[Theorem 1.1]{MPS}\label{theorem:main-result} We have
{\allowdisplaybreaks \begin{align}
C_{m,\ell}^{2K}(q)\pm C_{2K-m,\ell}^{2K}(q)
&=\frac{q^{s(m,\ell,2K)}}{J_{1}^3}\Big ( f_{K+1,K+1,1}(\pm q^{1+\frac{1}{2}(K+\ell)},q^{1+\frac{1}{2}(m+\ell)},q)
\label{equation:OP-split-1}\\
& \ \ \ \ \  \pm q^{\frac{1}{2}(K-\ell)}f_{K+1,K+1,1}(\pm q^{1+\frac{1}{2}(3K-\ell)},q^{1+K+\frac{1}{2}(m-\ell)},q)\Big).\notag
\end{align}}%
\end{theorem}
\begin{corollary} \cite[Corollary 1.2]{MPS}\label{corollary:f1K1-fKK1-A} We have
{\allowdisplaybreaks \begin{align}
C_{m,K}^{2K}(q)
&=\frac{q^{s(m,K,2K)}}{J_{1}^3}f_{K+1,K+1,1}(q^{K+1},q^{1+\frac{1}{2}(m+K)},q).
\label{equation:OP-2}
\end{align}}%
\end{corollary}

\begin{corollary} \cite[Corollary 1.3]{MPS}\label{corollary:f1K1-fKK1-B} For $K\equiv \ell \pmod 2$, we have
{\allowdisplaybreaks \begin{align}
C_{K,\ell}^{2K}(q)&=\frac{q^{s(K,\ell,2K)}}{J_{1}^3}
f_{K+1,K+1,1}(q^{1+\frac{1}{2}(K+\ell)},q^{1-\frac{1}{2}(K-\ell)},q).
\label{equation:OP-3}
\end{align}}%
\end{corollary}

In this paper, we will use the relation (\ref{equation:SW-fabc}), the double-sum formulas of \cite{HM}, and the new string function symmetries of \cite{MPS} to give general string functions identities.  In particular we will prove general string function identities for levels $N=1,2,3,4$:

\begin{theorem}\label{theorem:Level1-theorem} For $\ell \in\{0,1\}$ and $m\equiv \ell\pmod 2$, we have
\begin{equation}
q^{-\frac{1}{4}(m^2-\ell^2)}J_{1}^3\mathcal{C}_{m,\ell}^{1}(q)=J_{1}^2.
\end{equation}
\end{theorem}

\begin{theorem}\label{theorem:Level2-theorem}For $\ell \in\{0,1,2\}$, $0\le m < 4$, $m\equiv \ell\pmod 2$, we have
\begin{equation}
q^{-\frac{1}{8}(m^2-\ell^2)}J_{1}^3\mathcal{C}_{m,\ell}^{2}(q)
=
\begin{cases}
J_{1,2}\overline{J}_{3,8}
&\textup{if} \ \ell =0, m=0, \\
q^{\frac{1}{2}}J_{1,2}\overline{J}_{1,8}
&\textup{if} \ \ell =0, m=2, \\
J_{1}J_{2}
&\textup{if} \ \ell =1, m=1,3,\\
q^{\frac{1}{2}}J_{1,2}\overline{J}_{1,8}
&\textup{if} \ \ell =2, m= 0, \\
J_{1,2}\overline{J}_{3,8}
&\textup{if} \ \ell =2, m=2.
\end{cases}\label{equation:LevelN2-general}
\end{equation}
\end{theorem}

\begin{theorem} \label{theorem:Level3-theorem} For $\ell \in\{0,1,2, 3\}$, $0\le m < 6$, $m\equiv \ell \pmod 2$, we have
\begin{equation}
q^{-\frac{1}{12}(m^2-\ell^2)}J_{1}^3\mathcal{C}_{m,\ell}^{3}(q)
=
\begin{cases}
\theta_{0}(q)
&\textup{if} \ \ell =0, m=0, \\
\theta_{3}(q)
&\textup{if} \ \ell =0, m=2, 4,\\
\theta_{1}(q)
&\textup{if} \ \ell =1, m= 1, 5,\\
\theta_{2}(q)
&\textup{if} \ \ell =1, m= 3, \\
\theta_{2}(q)
&\textup{if} \ \ell =2, m= 0,\\
\theta_{1}(q)
&\textup{if} \ \ell =2, m= 2, 4,\\
\theta_{3}(q)
&\textup{if} \ \ell =3, m=1,5, \\
\theta_{0}(q)
&\textup{if} \ \ell =3, m=3,
\end{cases}
\label{equation:LevelN3-general}
\end{equation}
where
\begin{equation}
\theta_{i}(q):=
\begin{cases}
J_{1}\cdot ( J_{8,15}-qJ_{2,15})& \textup{for} \ i=0,\\
J_{1}J_{6,15}& \textup{for} \ i=1,\\
q^{\frac{1}{3}}J_{1}\cdot ( J_{11,15}+qJ_{1,15})& \textup{for} \ i=2,\\
q^{\frac{2}{3}}J_{1}J_{3,15}& \textup{for} \ i=3.
\end{cases}
\end{equation}
\end{theorem}

\begin{theorem} \label{theorem:Level4-theorem} For $\ell \in\{0,1,2, 3, 4\}$, $0\le m < 8$, and $m\equiv \ell \pmod 2$, we have
\begin{equation}
q^{-\frac{1}{16}(m^2-\ell^2)}J_{1}^3\mathcal{C}_{m,\ell}^{4}(q)
=
\begin{cases}
\theta_{0}(q)
&\textup{if} \ \ell =0, m= 0,\\
\theta_{2}(q)
&\textup{if} \ \ell =0, m= 2,6,\\
\theta_{1}(q)
&\textup{if} \ \ell =0, m= 4,\\
\theta_{3}(q)
&\textup{if} \ \ell =1, m= 1,7,\\
\theta_{4}(q)
&\textup{if} \ \ell =1, m=3,5,\\
\theta_{5}(q)
&\textup{if} \ \ell =2, m=0,4,\\
\theta_{6}(q)
&\textup{if} \ \ell =2, m= 2,6, \\
\theta_{4}(q)
&\textup{if} \ \ell =3, m=1,7,\\
\theta_{3}(q)
&\textup{if} \ \ell =3, m=3,5,\\
\theta_{1}(q)
&\textup{if} \ \ell =4, m= 0,\\
\theta_{2}(q)
&\textup{if} \ \ell =4, m= 2,6, \\
\theta_{0}(q)
&\textup{if} \ \ell =4, m=4,
\end{cases}
\label{equation:LevelN4-general}
\end{equation}
where
\begin{equation}
\theta_{i}(q):=
\begin{cases}
\frac{1}{2}\cdot (J_{1}\overline{J}_{3,6}+J_{1}J_{1,2})& \textup{for} \ i=0,\\
\frac{1}{2}\cdot (J_{1}\overline{J}_{3,6}-J_{1}J_{1,2})& \textup{for} \ i=1,\\
q^{\tfrac{3}{4}}J_{1}\overline{J}_{6,24}& \textup{for} \ i=2,\\
J_{1}\overline{J}_{3,8}& \textup{for} \ i=3,\\
q^{\tfrac{1}{2}}J_{1}\overline{J}_{1,8}& \textup{for} \ i=4,\\
q^{\tfrac{1}{4}}J_{1}\overline{J}_{1,6}& \textup{for} \ i=5,\\
J_{1,4}J_{6,12}& \textup{for} \ i=6.
\end{cases}
\end{equation}
\end{theorem}

We then demonstrate how the general string function identities give as special cases examples found in \cite[p. 220]{KP}:

\smallskip
\noindent Level 2:
{\allowdisplaybreaks \begin{subequations}
\begin{gather}
c_{20}^{20}-c_{02}^{20}=\eta(\tau)^{-2}\eta(\tau/2).
\label{equation:KP-2-List2-A}
\end{gather}
\end{subequations}}%
Level 3:
{\allowdisplaybreaks \begin{subequations}
\begin{gather}
c_{12}^{30}=\eta(\tau)^{-2}q^{27/40}\prod_{\substack{n\ge 1\\ n \not \equiv \pm 2 \pmod {5}}}(1-q^{3n}),
\label{equation:KP-3-List2-A}\\
c_{30}^{30}-c_{12}^{30}=\eta(\tau)^{-2}q^{1/120}\prod_{\substack{n\ge 1\\ n \not \equiv \pm 1 \pmod {5}}}(1-q^{n/3}),
\label{equation:KP-3-List2-B}\\
c_{21}^{21}-c_{03}^{21}=\eta(\tau)^{-2}q^{3/40}\prod_{\substack{n\ge 1\\ n \not \equiv \pm 2 \pmod {5}}}(1-q^{n/3}).
\label{equation:KP-3-List2-C}
\end{gather}
\end{subequations}}%

\noindent Level 4:
{\allowdisplaybreaks \begin{subequations}
\begin{gather}
c_{40}^{40}-2c_{22}^{40}+c_{04}^{40}+2c_{04}^{22}-2c_{22}^{22}
=\eta(\tau)^{-2}\eta(\tau/6)^{-1}\eta(\tau/12)^{2}.
\label{equation:KP-4-List2-B}
\end{gather}
\end{subequations}}%

As an interesting consequence, once one has Theorems \ref{theorem:Level1-theorem}, \ref{theorem:Level2-theorem}, \ref{theorem:Level3-theorem}, \ref{theorem:Level4-theorem} in mind, identities such as (\ref{equation:KP-2-List2-A}), (\ref{equation:KP-3-List2-B}), (\ref{equation:KP-3-List2-C}), and (\ref{equation:KP-4-List2-B}) become special cases of the classic theta function identity:
\begin{equation}
j(z;q)=\sum_{k=0}^{m-1}(-1)^k q^{\binom{k}{2}}z^k
j\big ((-1)^{m+1}q^{\binom{m}{2}+mk}z^m;q^{m^2}\big )\label{equation:j-split}.
\end{equation}
In particular identities (\ref{equation:KP-2-List2-A}), (\ref{equation:KP-3-List2-B}), (\ref{equation:KP-3-List2-C}), and (\ref{equation:KP-4-List2-B}) follow from the specializations $m=2,3,3,12$ respectively.

In Section \ref{section:prelim}, we recall background information on theta functions, Appell--Lerch functions, and Hecke-type double-sums.   In Section \ref{section:LevelN1-general} we present a proof of Theorem \ref{theorem:Level1-theorem}.  The proof is a corrected version of a sketch found in \cite[Example 1.3]{HM}.
In Section \ref{section:LevelN2-general} we present a new proof of Theorem \ref{theorem:Level2-theorem}.   In Section \ref{section:LevelN3-general} we present a new proof of Theorem \ref{theorem:Level3-theorem}.   In Section \ref{section:LevelN4-general} we present a new proof of Theorem \ref{theorem:Level4-theorem}.   In Section \ref{section:LevelN2-KP}, we use (\ref{equation:LevelN2-general}) to prove the level $N=2$ identities. In Section \ref{section:LevelN3-KP}, we use (\ref{equation:LevelN3-general}) to prove the level $N=3$ identities.  In Section \ref{section:LevelN4-KP}, we use (\ref{equation:LevelN4-general}) to prove the level $N=4$ identities.

Although there are general formulas for level $N$ string functions, see for example \cite{LP} and \cite[(6.5), (6.6)]{SW}, we emphasize that our methods here are new.  In particular, Kac and Peterson appeal to modularity to prove the string function identities  \cite[p. 220]{KP}.  They employ the transformation law for string functions under the full modular group, calculate the first few terms in the Fourier expansions of the string functions, and exploit the fact that a modular form vanishing at cusps to sufficiently high order is zero.  In the present paper, we use the relation (\ref{equation:SW-fabc}), the double-sum formulas of \cite{HM}.

We point out that the formula for $N=1$ is well-known, and the formula for $N=2$ is related to the Ising model in statistical mechanics.  Our formulations for $N=3$ and $N=4$ appear to be new; however, some of the pieces can be found in \cite[pp. 219-220]{KP}.

\section{Preliminaries}\label{section:prelim}

\subsection{Theta functions}
We collect some frequently encountered product rearrangements:
\begin{subequations}
\begin{gather}
\overline{J}_{0,1}=2\overline{J}_{1,4}=\frac{2J_2^2}{J_1},  \overline{J}_{1,2}=\frac{J_2^5}{J_1^2J_4^2},   J_{1,2}=\frac{J_1^2}{J_2},   \overline{J}_{1,3}=\frac{J_2J_3^2}{J_1J_6},\notag \\
 J_{1,4}=\frac{J_1J_4}{J_2},   J_{1,6}=\frac{J_1J_6^2}{J_2J_3},   \overline{J}_{1,6}=\frac{J_2^2J_3J_{12}}{J_1J_4J_6}.\notag
\end{gather}
\end{subequations}
Following from the definitions are the following general identities:
{\allowdisplaybreaks \begin{subequations}
\begin{gather}
j(q^n x;q)=(-1)^nq^{-\binom{n}{2}}x^{-n}j(x;q), \ \ n\in\mathbb{Z},\label{equation:j-elliptic}\\
j(x;q)=j(q/x;q)\label{equation:1.7},\\
j(x;q)={J_1}j(x;q^n)j(qx;q^n)\cdots j(q^{n-1}x;q^n)/{J_n^n} \ \ {\text{if $n\ge 1$,}}\label{equation:1.10}\\
j(x^n;q^n)={J_n}j(x;q)j(\zeta_nx;q)\cdots j(\zeta_n^{n-1}x;q^n)/{J_1^n},  \label{equation:1.12}
\end{gather}
\end{subequations}}%
if $n\ge 1$, $\zeta_n$ is a primitive $n$-th root of unity.

 A convenient form of the Weierstrass three-term relation for theta functions is,
\begin{proposition}\label{proposition:Weierstrass} For generic $a,b,c,d\in \mathbb{C}^*$
\begin{align*}
j(ac;q)&j(a/c;q)j(bd;q)j(b/d;q)\\
&=j(ad;q)j(a/d;q)j(bc;q)j(b/c;q)+b/c \cdot j(ab;q)j(a/b;q)j(cd;q)j(c/d;q).
\end{align*}
\end{proposition}

We collect several useful results about theta functions in terms of a proposition \cite{AH, H1, H2}: 
\begin{proposition}   For generic $x,y,z\in \mathbb{C}^*$ 
{\allowdisplaybreaks \begin{subequations}
\begin{gather}
j(x;q)j(y;q)=j(-xy;q^2)j(-qx^{-1}y;q^2)-xj(-qxy;q^2)j(-x^{-1}y;q^2),\label{equation:H1Thm1.1}\\
j(-x;q)j(y;q)-j(x;q)j(-y;q)=2xj(x^{-1}y;q^2)j(qxy;q^2),\label{equation:H1Thm1.2A}\\
j(-x;q)j(y;q)+j(x;q)j(-y;q)=2j(xy;q^2)j(qx^{-1}y;q^2).\label{equation:H1Thm1.2B}\\
j(x;q)j(y;q^n)=\sum_{k=0}^n(-1)^kq^{\binom{k}{2}}x^kj\big ((-1)^nq^{\binom{n}{2}+kn}x^ny;q^{n(n+1)}\big )j\big (-q^{1-k}x^{-1}y;q^{n+1} \big ).\label{equation:Thm1.3AH6}
\end{gather}
\end{subequations}}%
\end{proposition}

We finish this subsection with a series of lemmas.
\begin{lemma} \label{lemma:N4-theta-evaluation} We have
\begin{equation}
2\overline{J}_{1,6}\overline{J}_{1,3}
 +2 \overline{J}_{3,6}\overline{J}_{3,12}=\overline{J}_{0,1}\overline{J}_{0,2}.
\end{equation}
\end{lemma}
\begin{proof}  We specialize (\ref{equation:Thm1.3AH6}) to obtain
\begin{equation}
j(x;q)j(y;q^2)=\sum_{k=0}^2(-1)^kq^{\binom{k}{2}}x^kj\big (q^{1+2k}x^2y;q^{6}\big )j\big (-q^{1-k}x^{-1}y;q^{3}\big ).
\end{equation}
Hence
\begin{align*}
j(-1;q)j(-1;q^2)&=\sum_{k=0}^2q^{\binom{k}{2}}j\big (-q^{1+2k};q^{6}\big )j\big (-q^{1-k};q^{3}\big )\\
&=\overline{J}_{1,6}\overline{J}_{1,3}+\overline{J}_{3,6}\overline{J}_{0,3}
+q\overline{J}_{1,6}j(-q^{-1};q^3)\\
&=2\overline{J}_{1,6}\overline{J}_{1,3}+2\overline{J}_{3,6}\overline{J}_{3,12}.\qedhere
\end{align*}
\end{proof}

\begin{lemma} \label{lemma:levelN4-jsplit} We have
\begin{align*}
j(q^{1/12};q^{1/6})
&=\overline{J}_{12,24}+q^{3}\overline{J}_{0,24}
-2q^{\frac{3}{4}}\overline{J}_{6,24}
+2q^{\frac{1}{3}}\overline{J}_{8,24}+2q^{\frac{4}{3}}\overline{J}_{20,24}\\
& \ \ \ \ \  -2q^{\frac{1}{12}}\overline{J}_{10,24}-2q^{\frac{25}{12}}\overline{J}_{22,24}.
\end{align*}
\end{lemma}
\begin{proof}[Proof of Lemma \ref{lemma:levelN4-jsplit}]
From (\ref{equation:j-split}) with $m=12$, we have
{\allowdisplaybreaks \begin{align*}
j(x;q)&=j(-q^{66}z^{12};q^{144})-zj(-q^{78}z^{12};q^{144})
+qz^2j(-q^{90}z^{12};q^{144})\\
&\ \ \ \ \ -q^3z^3j(-q^{102}z^{12};q^{144})
+q^{6}z^4j(-q^{114}z^{12};q^{144})
-q^{10}z^5j(-q^{126}z^{12};q^{144})\\
&\ \ \ \ \ +q^{15}z^{6}j(-q^{138}z^{12};q^{144})
-q^{21}z^{7}j(-q^{150}z^{12};q^{144})
+q^{28}z^{8}j(-q^{162}z^{12};q^{144})\\
&\ \ \ \ \ -q^{36}z^{9}j(-q^{174}z^{12};q^{144})
+q^{45}z^{10}j(-q^{186}z^{12};q^{144})
-q^{55}z^{11}j(-q^{198}z^{12};q^{144}).
\end{align*}}%
Substituting $x\rightarrow q^{1/2}$ and using identities (\ref{equation:j-elliptic}) and (\ref{equation:1.7}) yields
{\allowdisplaybreaks \begin{align*}
j&(q^{1/2};q)\\
&=\overline{J}_{72,144}-q^{\frac{1}{2}}\overline{J}_{84,144}+q^{2}\overline{J}_{96,144}
-q^{\frac{9}{2}}\overline{J}_{108,144}+q^{8}\overline{J}_{120,144}-q^{\frac{25}{2}}\overline{J}_{132,144}\\
& \ \ \ \ \ +q^{18}\overline{J}_{144,144}-q^{\frac{49}{2}}\overline{J}_{156,144}+q^{32}\overline{J}_{168,144}
-q^{\frac{81}{2}}\overline{J}_{180,144}+q^{50}\overline{J}_{192,144}-q^{\frac{121}{2}}\overline{J}_{204,144}\\
&=\overline{J}_{72,144}-q^{\frac{1}{2}}\overline{J}_{84,144}+q^{2}\overline{J}_{96,144}
-q^{\frac{9}{2}}\overline{J}_{108,144}+q^{8}\overline{J}_{120,144}-q^{\frac{25}{2}}\overline{J}_{132,144}\\
& \ \ \ \ \ +q^{18}\overline{J}_{0,144}-q^{\tfrac{25}{2}}\overline{J}_{12,144}+q^{8}\overline{J}_{24,144}
-q^{\frac{9}{2}}\overline{J}_{36,144}+q^{2}\overline{J}_{48,144}-q^{\frac{1}{2}}\overline{J}_{60,144}\\
&=\overline{J}_{72,144}-2q^{\frac{1}{2}}\overline{J}_{84,144}+2q^{2}\overline{J}_{96,144}
-2q^{\frac{9}{2}}\overline{J}_{108,144}
 +2q^{8}\overline{J}_{120,144}\\
 & \ \ \ \ \  -2q^{\frac{25}{2}}\overline{J}_{132,144}
 +q^{18}\overline{J}_{0,144}.
\end{align*}}%
The result follow from identity (\ref{equation:1.7}) and the substitution $q\rightarrow q^{1/6}$.
\end{proof}

\subsection{Appell--Lerch functions}\label{section:prop-mxqz}

The Appell-Lerch function satisfies several functional equations and identities \cite{HM, Zw}:

\begin{proposition}  For generic $x,z\in \mathbb{C}^*$
{\allowdisplaybreaks \begin{subequations}
\begin{gather}
m(x,q,z)=m(x,q,qz),\label{equation:mxqz-fnq-z}\\
m(x,q,z)=x^{-1}m(x^{-1},q,z^{-1}),\label{equation:mxqz-flip}\\
m(qx,q,z)=1-xm(x,q,z),\label{equation:mxqz-fnq-x}\\
m(x,q,z_1)-m(x,q,z_0)=\frac{z_0J_1^3j(z_1/z_0;q)j(xz_0z_1;q)}{j(z_0;q)j(z_1;q)j(xz_0;q)j(xz_1;q)}.
\label{equation:changing-z-theorem}
\end{gather}
\end{subequations}}
\end{proposition}
\begin{corollary} \label{corollary:mxqz-eval} We have
\begin{align}
m(q,q^2,-1)&=1/2,\label{equation:mxqz-eval-a}\\
m(-1,q^2,q)&=0. \label{equation:mxqz-eval-b}
\end{align}
\end{corollary}

\subsection{Hecke-type double-sums}\label{section:double-sums}
We recall a few basic properties of Hecke-type double-sums.  We have a proposition and a corollary:

\begin{proposition} \cite[Proposition $6.3$]{HM}  For $x,y\in\mathbb{C}^{\star}$ and $R,S\in\mathbb{Z}$
\begin{align}
f_{a,b,c}(x,y,q)&=(-x)^{R}(-y)^Sq^{a\binom{R}{2}+bRS+c\binom{S}{2}}f_{a,b,c}(q^{aR+bS}x,q^{bR+cS}y,q) \label{equation:f-shift}\\
&\ \ \ \ +\sum_{m=0}^{R-1}(-x)^mq^{a\binom{m}{2}}j(q^{mb}y;q^c)+\sum_{m=0}^{S-1}(-y)^mq^{c\binom{m}{2}}j(q^{mb}x;q^a).\notag
\end{align}
\end{proposition}
We also have the property \cite[$(6.2)$]{HM}:
\begin{equation}
f_{a,b,c}(x,y,q)=-\frac{q^{a+b+c}}{xy}f_{a,b,c}(q^{2a+b}/x,q^{2c+b}/y,q).\label{equation:f-flip}
\end{equation}

In order to state the double-sum formulas that we will be using, we introduce the useful
\begin{align}
g_{1,b,1}(x,y,q,z_1,z_0)
&:=j(y;q)m\Big (q^{\binom{b+1}{2}-1}x(-y)^{-b},q^{b^2-1},z_1\Big )\label{equation:mdef-2}\\
&\ \ \ \ \ + j(x;q)m\Big (q^{\binom{b+1}{2}-1}y(-x)^{-b},q^{b^2-1},z_0\Big ).\notag
\end{align}

In \cite[Theorem 1.3]{HM}, we specialize $n=1$, to have 
\begin{theorem}  \label{theorem:masterFnp} Let $p$ be a positive integer.  For generic $x,y\in \mathbb{C}^*$
\begin{align*}
f_{1,p+1,1}(x,y,q)=g_{1,p+1,1}(x,y,q,-1,-1)+\frac{1}{\overline{J}_{0,p(2+p)}}\cdot \theta_{p}(x,y,q),
\end{align*}
where
\begin{align*}
&\theta_{p}(x,y,q):=\sum_{r=0}^{p-1}\sum_{s=0}^{p-1}q^{\binom{r}{2}+(1+p) (r) (s+1 )+\binom{s+1}{2}}  (-x)^{r}(-y)^{s+1}\notag\\
 & \cdot   \frac{J_{p^2(2+p)}^3j(-q^{p(s-r)}x/y;q^{p^2})j(q^{p(2+p)(r+s)+p(1+p)}x^py^p;q^{p^2(2+p)})}{j(q^{p(2+p)r+p(1+p)/2}(-y)^{1+p}/(-x);q^{p^2(2+p)})j(q^{p(2+p)s+p(1+p)/2}(-x)^{1+p}/(-y);q^{p^2(2+p)})}.\notag
\end{align*}
\end{theorem}

The specialization for $p=1$ will be of importance.  It is just (\ref{equation:f121}):
\begin{corollary} We have
\begin{align}
&\Big (\sum_{r,s\ge 0}-\sum_{r,s<0}\Big )(-1)^{r+s}x^ry^sq^{\binom{r}{2}+2rs+\binom{s}{2}}\\
&\ \ \ \ \ =j(y;q)m\Big (\frac{q^2x}{y^2},q^3,-1\Big )+j(x;q)m\Big (\frac{q^2y}{x^2},q^3,-1\Big )
- \frac{yJ_{3}^3j(-x/y;q)j(q^2xy;q^3)}{\overline{J}_{0,3}j(-qy^2/x,-qx^2/y;q^3)}.\notag 
\end{align}
\end{corollary}

\noindent For another useful result, we specialize \cite[Theorem 1.4]{HM} to $a=b=n$, $c=1$.  
\begin{theorem} \label{theorem:main-acdivb} Let $n$ be a positive integer.  Then
\begin{align*}
& f_{n,n,1}(x,y,q)=h_{n,n,1}(x,y,q,-1,-1)-\frac{1}{\overline{J}_{0,n-1}\overline{J}_{0,n^2-n}}\cdot \theta_{n}(x,y,q),
\end{align*}
where 
\begin{align*}
h_{n,n,1}(x,y,q,z_1,z_0):&=j(x;q^n)m\Big( -q^{n-1}yx^{-1},q^{n-1},z_1 \Big )\\
& \ \ \ \ \ \ +j(y;q)m\Big( q^{\binom{n}{2}}x(-y)^{-n},q^{n^2-n},z_0 \Big ),
\end{align*}
and
\begin{align*}
&\theta_{n}(x,y,q):=\sum_{d=0}^{n-1}
q^{(n-1)\binom{d+1}{2}}j\big (q^{(n-1)(d+1)}y;q^{n}\big )  j\big (-q^{n(n-1)-(n-1)(d+1)}xy^{-1};q^{n(n-1)}\big ) \\
& \cdot \frac{J_{n(n-1)}^3j\big (q^{ \binom{n}{2}+(n-1)(d+1)}(-y)^{1-n};q^{n(n-1)}\big )}
{j\big (-q^{\binom{n}{2}}x(-y)^{-n};q^{n(n-1)})j(q^{(n-1)(d+1)}x^{-1}y;q^{n(n-1)}\big )}.
\end{align*}
\end{theorem}

Theorem \ref{theorem:main-acdivb} has the following specializations.

\begin{corollary}  \label{corollary:f221-expansion}  We have
{\allowdisplaybreaks \begin{align}
f_{2,2,1}(x,y,q)&=h_{2,2,1}(x,y,q,-1,-1)\label{equation:f221-id}\\
&\ \ \ \ -\sum_{d=0}^{1}
\frac{q^{\binom{d+1}{2}}j\big (q^{1+d}y;q^{2}\big )  j\big (-q^{1-d}x/y;q^{2}\big ) 
J_{2}^3j\big (-q^{2+d}/y;q^{2}\big )}
{4\overline{J}_{1,4}\overline{J}_{2,8}j\big (-qx/y^{2};q^2\big )j\big (q^{1+d}y/x;q^2\big )},\notag
\end{align}}%
where
\begin{align}
h_{2,2,1}(x,y,q,-1,-1)=j(x;q^2)m(-qx^{-1}y,q,-1)+j(y;q)m(qxy^{-2},q^2,-1).\label{equation:g221-id}
\end{align}
\end{corollary}

\begin{corollary}  \label{corollary:f331-expansion}  We have
{\allowdisplaybreaks \begin{align}
f_{3,3,1}(x,y,q)&=h_{3,3,1}(x,y,q,-1,-1)\label{equation:f331-id}\\
&\ \ \ \ -\sum_{d=0}^{2}
\frac{q^{d(d+1)}j\big (q^{2+2d}y;q^{3}\big )  j\big (-q^{4-2d}x/y;q^{6}\big ) J_{6}^3j\big (q^{5+2d}/y^{2};q^{6}\big )}
{4\overline{J}_{2,8}\overline{J}_{6,24}j\big (q^{3}x/y^{3};q^6\big )j\big (q^{2+2d}y/x;q^6\big )},\notag
\end{align}}%
where
\begin{align}
h_{3,3,1}(x,y,q,-1,-1)=j(x;q^3)m(-q^2x^{-1}y,q^2,-1)+j(y;q)m(-q^3xy^{-3},q^6,-1).\label{equation:g331-id}
\end{align}
\end{corollary}

In Theorem \ref{theorem:masterFnp}, we set $z_1=z_0=-1$ in the Appell-Lerch expression (\ref{equation:mdef-2}).  For examples where $p=2,3$, we can set $z_1=z_0^{-1}=y/x$ to reduce the number of theta quotients.    For example, we can specialize \cite[Theorem 1.9]{HM} to $n=1$ to have
\begin{theorem} \label{theorem:genfn2} For generic $x,y\in\mathbb{C}^*$
\begin{align*}
f_{1,3,1}(x,y,q)=g_{1,3,1}(x,y,q,y/x,x/y)-\Theta_{1,2}(x,y,q),
\end{align*}
where
\begin{align*}
\Theta_{1,2}(x,y,q):=\frac{qxyJ_{2,4}J_{8,16}j(q^{3}xy;q^{8})j(q^{2}/x^2y^2;q^{16})}
{j(-q^{3}x^2;q^8)j(-q^{3}y^2;q^{8})}.
\end{align*}
\end{theorem}
We can also specialize \cite[Theorem 1.10]{HM} to $n=1$ to have
\begin{theorem} \label{theorem:genfn3} For generic $x,y\in\mathbb{C}^{*}$
\begin{align*}
f_{1,4,1}(x,y,q)=g_{1,4,1}(x,y,q,y/x,x/y)-\Theta_{1,3}(x,y,q),
\end{align*}
where
{\allowdisplaybreaks \begin{align*}
\Theta_{1,3}&(x,y,q):=\frac{qxyJ_{3}J_{15}j(q^{2}x,q^{2}y;q^{5})}
 {J_{5}^2j(q^{6}x^3;q^{15})j(q^{6}y^3;q^{15})}\\
&\cdot  \Big \{ j(q^{11}x^2y;q^{15})j(q^{11}xy^2;q^{15})  -{q^{4}}{xy}j(q^{16}x^2y;q^{15})j(q^{16}xy^2;q^{15})\Big \}.
\end{align*}}
\end{theorem}

We can also specialize \cite[Theorem 1.11]{HM} to $n=1$ to have
\begin{theorem}  \label{theorem:genfn4} Let $n$ be a positive odd integer. For generic $x,y\in\mathbb{C}^{*}$
{\allowdisplaybreaks \begin{align*}
f_{1,5,1}(x,y,q)=g_{1,5,1}(x,y,q,y/x,x/y)-\Theta_{1,4}(x,y,q),
\end{align*}
where
\begin{align}
g_{1,5,1}(x,y,q,y/x,x/y)
&:=j(y;q)m\Big (-q^{14}xy^{-5},q^{24},y/x\Big )\\
&\ \ \ \ \ + j(x;q)m\Big (-q^{14}yx^{-5},q^{24},x/y\Big ).\notag
\end{align}
and
\begin{align*}
\Theta_{1,4}(x,y,q):=&\frac{qxy}{j(-q^{10}x^4;q^{24})j(-q^{10}y^4;q^{24})}\Big \{ J_{4,16}  S_1-qJ_{8,16} S_2\Big \},
\end{align*}
with
\begin{align*}
S_1:&=\frac{j(q^{22}x^2y^2,-q^{12}y/x;q^{24})j(q^{5}xy;q^{12})}{J_{12}^3J_{48}} \\
& \cdot  \Big \{j(-q^{10}x^2y^2,q^{12}y^2/x^2;q^{24})J_{24}^2\\
&\ \ \ \ +\frac{q^{5}x^2 j(-q^{22}x^2y^2;q^{24})j(q^{12}y/x,-y/x;q^{24})^2}{J_{24}}\Big \},\\
S_2:&=\frac{j(q^{10}x^2y^2,-y/x;q^{24})j(q^{11}xy;q^{12})}{J_{12}^2} \\
& \cdot  \Big \{\frac{q^{2}j(-q^{10}x^2y^2,q^{12}y^2/x^2;q^{24})J_{48}}{yJ_{24} }\\ 
&\ \ \ \ +\frac{qxj(-q^{22}x^2y^2;q^{24})j(q^{24}y^2/x^2;q^{48})^2}{J_{48} }\Big \}.
\end{align*}}%
\end{theorem}

\begin{proposition}\cite[Proposition 8.1]{HM}\label{proposition:prop-singshift}  Let $\ell\in \mathbb{Z}$, $p\in \{ 1,2,3,4\}$, $n\in \mathbb{N}$ with $(n,p)=1$.  For generic $x,y\in\mathbb{C}^*$
\begin{equation*}
f_{1,1+p,1}(x,y,q)=g_{1,1+p,1}(x,y,q,q^{\ell p}y/x,q^{-\ell p}x/y)-
(-x)^{\ell}q^{\binom{\ell}{2}}\Theta_{1,p}(q^{\ell }x,q^{\ell (1+p)}y,q).
\end{equation*}
\end{proposition}


\subsection{The general integral-level string function}\label{section:Level-gen-N}  

We recall the notation 
\begin{equation}
s(m,\ell,N):=\frac{(\ell+1)^2}{4(N+2)}-\frac{m^2}{4N}-\frac{1}{8}
\end{equation}
and the fact that \cite{SW}:
\begin{equation}
\mathcal{C}_{m,\ell}^{N}(q):=q^{-s(m,\ell,N)}C_{m,\ell}^{N}(q).
\end{equation}
The following is a straightforward consequence of the symmetry relations (\ref{equation:string-symmetry-1}) - (\ref{equation:string-symmetry-3}).
\begin{lemma} \label{lemma:masterLemma} If
\begin{equation}
f_{m,\ell}^{N}(q):=q^{-\tfrac{1}{4N}(m^2-\ell^2)}\mathcal{C}_{m,\ell}^{N}(q),
\end{equation}
then
\begin{equation}
f_{m,\ell}^{N}(q)=f_{-m,\ell}^{N}(q)=f_{2N-m,\ell}^{N}(q)=f_{N-m,N-\ell}^{N}(q).
\end{equation}
\end{lemma}

\section{Computing the general level $N=1$ string function}  \label{section:LevelN1-general}

Let us set $N=1$.  Here $\ell\in \{0,1\}$, $0\le m<2$, $m\equiv \ell \pmod 2$.   The proof of Theorem \ref{theorem:Level1-theorem} follows from a lemma, whose proof we do now.

\begin{lemma}\label{lemma:f121-evaluations} We have
\begin{equation}
f_{1,2,1}(q,q,q)=J_{1}^2.\label{equation:f121-0}
\end{equation}
\end{lemma}

\begin{proof}[Proof of Lemma \ref{lemma:f121-evaluations}]  This is just the calculation found in (\ref{equation:f121-calculation}).
\end{proof}

\begin{proof}[Proof of Theorem \ref{theorem:Level1-theorem}]  For $N=1$ we only need to be concerned with $\ell\in\{0,1\}$, $0\le m <2$, $m\equiv \ell \pmod 2$.  Which means we only need to compute the string functions for the $(\ell,m)$ tuples $(0,0)$ and $(1,1)$.  For the case $(\ell,m)=(0,0)$, we have
\begin{equation*}
q^{-\frac{1}{4}(m^2-\ell^2)}J_{1}^3\mathcal{C}_{m,\ell}^{1}(q)
=J_{1}^3\mathcal{C}_{0,0}^{1}(q)
=f_{1,2,1}(q,q,q)
=J_{1}^2,
\end{equation*}
where the last equality follows from Lemma \ref{lemma:f121-evaluations}.  The case $(\ell,m)=(1,1)$ follows from Lemma \ref{lemma:masterLemma}.
\end{proof}

\section{Computing the general level $N=2$ string function}\label{section:LevelN2-general} 
Here $N=2$, $\ell\in\{0,1,2\}$, $0\le m < 4$, and $m\equiv \ell \pmod 2$.    We begin with a proposition.

\begin{proposition}\label{proposition:f131-evaluations} We have
{\allowdisplaybreaks \begin{align}
f_{1,3,1}(q,q,q)&=J_{1,2}\overline{J}_{3,8},\label{equation:f131-0}\\
f_{1,3,1}(q^2,q,q)&=J_{1}J_{2},\label{equation:f131-1}\\
f_{1,3,1}(q^2,q^2,q)&=J_{1,2}\overline{J}_{1,8}.\label{equation:f131-2}
\end{align}}
\end{proposition}

\begin{proof}[Proof of Proposition \ref{proposition:f131-evaluations}]  We use Theorem \ref{theorem:genfn2} and Proposition \ref{proposition:prop-singshift} with the specialization $n=1$:
\begin{align}
f_{1,3,1}(x,y,q)&=j(y;q)m(-q^5x/y^3,q^8,q^{2k}y/x)+j(x;q)m(-q^5y/x^3,q^8,x/q^{2k}y)
\label{equation:f131-genk}\\
&\ \ \ \ \ -(-1)^kq^{4k+1+\binom{k}{2}}x^{k+1}y
\frac{J_{2,4}J_{8,16}j(q^{4k+3}xy;q^{8})j(q^{8k+14}x^2y^2;q^{16})}
{j(-q^{2k+3}x^2;q^{8})j(-q^{6k+3}y^2;q^{8 })}.\notag
\end{align}  

We prove (\ref{equation:f131-0}).  We use (\ref{equation:f131-genk}) with $k=1$ to have
\begin{align*}
f_{1,3,1}(q,q,q)&=j(q;q)m(-q^3,q^8,q^{2})+j(q;q)m(-q^3,q^8,q^{-2})\\
&\ \ \ \ \ +q^{8}
\frac{J_{2,4}J_{8,16}j(q^{9};q^{8})j(q^{26};q^{16})}
{j(-q^{7};q^{8})j(-q^{11};q^{8 })}.
\end{align*}
The $m(x,q,z)$ terms are defined, and their theta coefficients are both zero.  Hence
\begin{equation*}
f_{1,3,1}(q,q,q)
=q^{8}
\frac{J_{2,4}J_{8,16}j(q^{9};q^{8})j(q^{26};q^{16})}
{j(-q^{7};q^{8})j(-q^{11};q^{8 })}=\frac{J_{2,4}J_{8,16}J_{1,8}J_{10,16}}
{\overline{J}_{7,8}\overline{J}_{3,8}}
=J_{1,2}\overline{J}_{3,8},
\end{equation*}
where we have used (\ref{equation:j-elliptic}), (\ref{equation:1.7}), and elementary product rearrangements.

We prove (\ref{equation:f131-1}).  We use (\ref{equation:f131-genk}) with $k=0$ to have
{\allowdisplaybreaks \begin{align*}
f_{1,3,1}(q^2,q,q)&=j(q;q)m(-q^4,q^8,q^{-1})+j(q^2;q)m(-1,q^8,q)\\
&\ \ \ \ \ -q^4
\frac{J_{2,4}J_{8,16}j(q^{6};q^{8})j(q^{20};q^{16})}
{j(-q^{7};q^{8})j(-q^{5};q^{8 })}\\
&=\frac{J_{2,4}J_{8,16}J_{2,8}J_{4,16}}
{\overline{J}_{1,8}\overline{J}_{5,8}}\\
&=J_{1}J_{2},
\end{align*}}%
where we have used (\ref{equation:j-elliptic}), (\ref{equation:1.7}), and elementary product rearrangements.

We prove (\ref{equation:f131-2}).  We use (\ref{equation:f131-genk}) with $k=1$ to have
{\allowdisplaybreaks \begin{align*}
f_{1,3,1}(q^2,q^2,q)&=j(q^2;q)m(-q,q^8,q^{2})+j(q^2;q)m(-q,q^8,q^{-2})\\
&\ \ \ \ \ +q^{11}
\frac{J_{2,4}J_{8,16}j(q^{11};q^{8})j(q^{30};q^{16})}
{j(-q^{9};q^{8})j(-q^{13};q^{8 })}\\
&=\frac{J_{2,4}J_{8,16}J_{3,8}J_{2,16}}
{\overline{J}_{1,8}\overline{J}_{5,8}}\\
&=J_{1,2}\overline{J}_{1,8},
\end{align*} }%
where we have used (\ref{equation:j-elliptic}), (\ref{equation:1.7}), and elementary product rearrangements.
\end{proof}

\begin{proof}[Proof of Theorem \ref{theorem:Level2-theorem}]   For $N=2$ we only need to be concerned with $\ell\in\{0,1, 2\}$, $0\le m <4$, $m\equiv \ell \pmod 2$.  Which means we only need to compute the string functions for the $(\ell,m)$-tuples $(0,0), (0,2), (1,1), (1,3), (2,0), (2,2)$.  Because of Lemma \ref{lemma:masterLemma}, we only need to be concerned with the $(\ell,m)$-tuples $(0,0), (1,1), (2,0)$.

For $(\ell,m)=(0,0)$, we have
\begin{equation*}
q^{-\frac{1}{8}(m^2-\ell^2)}J_{1}^3\mathcal{C}_{m,\ell}^{1}(q)
=J_{1}^3\mathcal{C}_{0,0}^{1}(q)
=f_{1,3,1}(q,q,q)
=J_{1,2}\overline{J}_{3,8},
\end{equation*}
where the last equality follows from identity (\ref{equation:f131-0}).  For $(\ell,m)=(1,1)$, we have
\begin{equation*}
q^{-\frac{1}{8}(m^2-\ell^2)}J_{1}^3\mathcal{C}_{m,\ell}^{1}(q)
=J_{1}^3\mathcal{C}_{1,1}^{1}(q)
=f_{1,3,1}(q^2,q,q)
=J_{1}J_{2},
\end{equation*}
where the last equality follows from identity (\ref{equation:f131-1}).  For $(\ell,m)=(2,0)$, we have
\begin{equation*}
q^{-\frac{1}{8}(m^2-\ell^2)}J_{1}^3\mathcal{C}_{m,\ell}^{1}(q)
=q^{\tfrac{1}{2}}J_{1}^3\mathcal{C}_{2,0}^{1}(q)
=q^{\tfrac{1}{2}}f_{1,3,1}(q^2,q^2,q)
=q^{\tfrac{1}{2}}J_{1,2}\overline{J}_{1,8},
\end{equation*}
where the last equality follows from identity (\ref{equation:f131-2}).
\end{proof}


\section{Computing the general level $N=3$ string function}\label{section:LevelN3-general} 
Here $N=3$, $\ell\in\{0,1,2, 3\}$, $0\le m < 6$, and $m\equiv \ell \pmod 2$.    The proof of Theorem \ref{theorem:Level3-theorem} follows from a proposition.

\begin{proposition}\label{proposition:f141-evaluations} We have
\begin{align}
f_{1,4,1}(q,q,q)&=J_{1}\cdot ( J_{8,15}-qJ_{2,15}),\label{equation:f141-evaluation-0}\\
f_{1,4,1}(q^2,q,q)&=J_{1}J_{6,15},\label{equation:f141-evaluation-1}\\
f_{1,4,1}(q^2,q^2,q)&=J_{1}\cdot ( J_{11,15}+qJ_{1,15}),\label{equation:f141-evaluation-2}\\
f_{1,4,1}(q^3,q^2,q)&=J_{1}J_{3,15}.\label{equation:f141-evaluation-3}
\end{align}
\end{proposition}

\begin{proof}[Proof of Proposition \ref{proposition:f141-evaluations}]  We use Theorem \ref{theorem:genfn3} and Proposition \ref{proposition:prop-singshift}:
{\allowdisplaybreaks \begin{align}
f_{1,4,1}(x,y,q)&=j(y;q)m(q^9x/y^4,q^{15},q^{3k}y/x)+j(x;q)m(q^9y/x^4,q^{15},x/q^{3k}y)\label{equation:f141-genk}\\
& \ \ \ \ \ -(-x)^kq^{\binom{k}{2}}\frac{q^{5k+1}xyJ_{3}J_{15}j(q^{2+k}x;q^{5})j(q^{2+4k}y;q^5)}
{J_{5}^2j(q^{6+3k}x^3;q^{15})j(q^{6+12k}y^3;q^{15})}\notag\\
& \ \ \ \ \ \cdot \Big [ j(q^{11+6k}x^2y;q^{15})j(q^{11+9k}xy^2;q^{15})\notag\\
& \ \ \ \ \ \ \ \ \ \ -q^{4+5k}xyj(q^{16+6k}x^2y;q^{15})j(q^{16+9k}xy^2;q^{15})\Big ]. \notag
\end{align}

We prove (\ref{equation:f141-evaluation-0}). In (\ref{equation:f141-genk}), set $k=1$
\begin{align*}
f_{1,4,1}(q,q,q)&=j(q;q)m(q^6,q^{15},q^{3})+j(q;q)m(q^6,q^{15},q^{-3})\\
& \ \ \ \ \ +q^9\frac{J_{3}J_{15}j(q^{4};q^{5})j(q^{7};q^5)}
{J_{5}^2j(q^{12};q^{15})j(q^{21};q^{15})}\\
& \ \ \ \ \ \cdot \Big [ j(q^{20};q^{15})j(q^{23};q^{15})
 -q^{11}j(q^{25};q^{15})j(q^{28};q^{15})\Big ]\\
&=\frac{q^{13}J_{3}J_{15}J_{1,5}J_{2,5}}
{J_{5}^2J_{3,15}J_{6,15}}
 \cdot \Big [ q^{-13}J_{5}J_{8,15}
-q^{-12}J_{5}J_{2,15}\Big ],
\end{align*}}%
where we have used (\ref{equation:j-elliptic}).  Simplifying, we have
\begin{equation*}
f_{1,4,1}(q,q,q)
=\frac{J_{3}J_{15}J_{1,5}J_{2,5}}
{J_{5}J_{3,15}J_{6,15}}
 \cdot \Big [ J_{8,15}
-qJ_{2,15}\Big ]=J_1
 \cdot \Big [ J_{8,15}
-qJ_{2,15}\Big ],
\end{equation*}
where we have twice used the product rearrangement $J_{1,5}J_{2,5}=J_1J_5$.

For (\ref{equation:f141-evaluation-1}), we recall (\ref{equation:f141-genk}) and set $k=0$ to have
{\allowdisplaybreaks \begin{align*}
f_{1,4,1}(q^2,q,q)&=j(q;q)m(q^7,q^{15},q^{-1})+j(q^2;q)m(q^2,q^{15},q)\\
& \ \ \ \ \ -\frac{q^4J_{3}J_{15}j(q^{4};q^{5})j(q^3;q^5)}
{J_{5}^2j(q^{12};q^{15})j(q^{9};q^{15})}\\
& \ \ \ \ \ \cdot \Big [ j(q^{16};q^{15})j(q^{15};q^{15})
-q^{7}j(q^{21};q^{15})j(q^{20};q^{15})\Big ]\\
&=\frac{J_{3}J_{15}J_{4,5}J_{3,5}J_{6,15}J_{5}}
{J_{5}^2J_{3,15}J_{6,15}}\\
&=J_{1}J_{6,15},
\end{align*}}%
where for the last equality we used (\ref{equation:j-elliptic}), (\ref{equation:1.7}), and elementary product rearrangements.

We prove (\ref{equation:f141-evaluation-2}).  In (\ref{equation:f141-genk}), we set $k=2$:
{\allowdisplaybreaks \begin{align*}
f_{1,4,1}(q^2,q^2,q)&=j(q^2;q)m(q^3,q^{15},q^{6})+j(q^2;q)m(q^3,q^{15},q^{-6})\\
& \ \ \ \ \ -q^{20}\frac{J_{3}J_{15}j(q^{6};q^{5})j(q^{12};q^5)}
{J_{5}^2j(q^{18};q^{15})j(q^{36};q^{15})}\\
& \ \ \ \ \ \cdot \Big [ j(q^{29};q^{15})j(q^{35};q^{15})
 -q^{18}j(q^{34};q^{15})j(q^{40};q^{15})\Big ]\\
&=\frac{J_{3}J_{15}J_{1,5}J_{2,5}}
{J_{5}^2J_{3,15}J_{6,15}}
 \cdot \Big [ qJ_{14,15}J_{5}
 +J_{4,15}J_{5}\Big ]\\
&=J_{1} \cdot \Big [ J_{4,15}+ qJ_{14,15} \Big ].
\end{align*}}%

For (\ref{equation:f141-evaluation-3}), we take (\ref{equation:f141-genk}) and set $k=1$:
{\allowdisplaybreaks \begin{align*}
f_{1,4,1}(q^3,q^2,q)&=j(q^2;q)m(q^4,q^{15},q^{2})+j(q^3;q)m(q^{-1},q^{15},q^{-2})\\
& \ \ \ \ \ +q^{14}\frac{J_{3}J_{15}j(q^{6};q^{5})j(q^{8};q^5)}
{J_{5}^2j(q^{18};q^{15})j(q^{24};q^{15})}\\
& \ \ \ \ \ \cdot \Big [ j(q^{25};q^{15})j(q^{27};q^{15})
-q^{14}j(q^{30};q^{15})j(q^{32};q^{15})\Big ]\\
&=\frac{J_{3}J_{15}J_{1,5}J_{3,5}}
{J_{5}^2J_{3,15}J_{9,15}}
 \cdot  J_{5}J_{12,15}\\
&=J_{1}J_{3,15}. \qedhere
\end{align*}}%

\end{proof}

\begin{proof}[Proof of Theorem \ref{theorem:Level3-theorem}]   For $N=3$ we only need to be concerned with $\ell\in\{0,1, 2, 3\}$, $0\le m <6$, $m\equiv \ell \pmod 2$.  Because of Lemma \ref{lemma:masterLemma}, we only need to be concerned with the $(\ell,m)$-tuples $(0,0), (1,1), (2,0), (3,1)$.

For $(\ell,m)=(0,0)$, we have
\begin{equation*}
q^{-\frac{1}{12}(m^2-\ell^2)}J_{1}^3\mathcal{C}_{m,\ell}^{1}(q)
=J_{1}^3\mathcal{C}_{0,0}^{1}(q)
=f_{1,4,1}(q,q,q)
=J_{1}\cdot ( J_{8,15}-qJ_{2,15}),
\end{equation*}
where the last equality follows from identity (\ref{equation:f141-evaluation-0}).  For $(\ell,m)=(1,1)$, we have
\begin{equation*}
q^{-\frac{1}{12}(m^2-\ell^2)}J_{1}^3\mathcal{C}_{m,\ell}^{1}(q)
=J_{1}^3\mathcal{C}_{1,1}^{1}(q)
=f_{1,4,1}(q^2,q,q)
=J_{1}J_{6,15},
\end{equation*}
where the last equality follows from identity (\ref{equation:f141-evaluation-1}).  For $(\ell,m)=(2,0)$, we have
\begin{equation*}
q^{-\frac{1}{12}(m^2-\ell^2)}J_{1}^3\mathcal{C}_{m,\ell}^{1}(q)
=q^{\tfrac{1}{3}}J_{1}^3\mathcal{C}_{2,0}^{1}(q)
=q^{\tfrac{1}{3}}f_{1,4,1}(q^2,q^2,q)
=q^{\tfrac{1}{3}}J_{1}\cdot ( J_{11,15}+qJ_{1,15}),
\end{equation*}
where the last equality follows from identity (\ref{equation:f141-evaluation-2}).  For $(\ell,m)=(3,1)$, we have
\begin{equation*}
q^{-\frac{1}{12}(m^2-\ell^2)}J_{1}^3\mathcal{C}_{m,\ell}^{1}(q)
=q^{\tfrac{2}{3}}J_{1}^3\mathcal{C}_{3,1}^{1}(q)
=q^{\tfrac{2}{3}}f_{1,4,1}(q^3,q^2,q)
=q^{\tfrac{2}{3}}J_{1}J_{3,15},
\end{equation*}
where the last equality follows from identity (\ref{equation:f141-evaluation-3}).  
\end{proof}


\section{Computing the general level $N=4$ string function}\label{section:LevelN4-general} 

\begin{proof}[Proof of Theorem \ref{theorem:Level4-theorem}]   For $N=4$ we only need to be concerned with $\ell\in\{0,1, 2, 3, 4\}$, $0\le m <8$, $m\equiv \ell \pmod 2$.  Because of Lemma \ref{lemma:masterLemma}, we only need to be concerned with the $(\ell,m)$-tuples $(0,0), (0,4), (0,2), (1,1), (1,3), (2,0), (2,2)$.  For each $(\ell,m)$-tuple, one uses Proposition \ref{proposition:N4-string-evaluations} to compute
\begin{equation*}
q^{-\tfrac{1}{16}(m^2-\ell^2)}J_{1}^3\mathcal{C}_{m,\ell}^{N}(q).\qedhere
\end{equation*}
\end{proof}

\begin{proposition} \label{proposition:N4-Hecke-evaluations} We have
{\allowdisplaybreaks \begin{subequations}
\begin{gather}
f_{3,3,1}(-q^2,q,q)-qf_{3,3,1}(-q^4,q^3,q)=J_{1}J_{1,2},\label{equation:N4-Hecke-evaluation-0}\\
f_{3,3,1}(q^2,q,q)+qf_{3,3,1}(q^4,q^3,q)=J_{1}\overline{J}_{3,6},\label{equation:N4-Hecke-evaluation-1}\\
f_{1,5,1}(q^2,q^2,q)=J_{1}\overline{J}_{1,6},\label{equation:N4-Hecke-evaluation-2}\\
f_{3,3,1}(q^3,q,q)=J_{1,4}J_{6,12},\label{equation:N4-Hecke-evaluation-3}\\
f_{1,5,1}(q^2,1,q)=qJ_{1}\overline{J}_{6,24},\label{equation:N4-Hecke-evaluation-4}\\
f_{3,3,1}( q^{5},q^4,q^2)+ qf_{3,3,1}( q^{7},q^6,q^2)=J_{2}\overline{J}_{1,4},\label{equation:N4-Hecke-evaluation-5}\\
f_{3,3,1}( -q^{5},q^4,q^2)- qf_{3,3,1}( -q^{7},q^6,q^2)=J_{2}J_{1,4}.\label{equation:N4-Hecke-evaluation-6}
\end{gather}
\end{subequations}}%
\end{proposition}

\begin{proposition} \label{proposition:N4-string-evaluations}  We have
{\allowdisplaybreaks \begin{subequations}
\begin{gather}
\mathcal{C}_{0,0}^{4}(q)=\frac{1}{2}(J_{1}\overline{J}_{3,6}+J_{1}J_{1,2}),\label{equation:N4-string-evaluation-0} \\
\mathcal{C}_{4,0}^{4}(q)=q\frac{1}{2}(J_{1}\overline{J}_{3,6}-J_{1}J_{1,2}),\label{equation:N4-string-evaluation-1}\\
\mathcal{C}_{2,0}^{4}(q)=qJ_{1}\overline{J}_{6,24},\label{equation:N4-string-evaluation-2}\\
\mathcal{C}_{1,1}^{4}(q)=J_{1}\overline{J}_{3,8},\label{equation:N4-string-evaluation-3}\\
\mathcal{C}_{3,1}^{4}(q)= J_{1}\overline{J}_{1,8},\label{equation:N4-string-evaluation-4}\\
\mathcal{C}_{0,2}^{4}(q)= J_{1}\overline{J}_{1,6},\label{equation:N4-string-evaluation-5}\\
\mathcal{C}_{2,2}^{4}(q)= J_{1,4}J_{6,12}.\label{equation:N4-string-evaluation-6}
\end{gather}
\end{subequations}}%

\end{proposition}

\begin{proof}[Proof of Proposition \ref{proposition:N4-Hecke-evaluations}]   We recall 
\begin{equation}
C_{m,\ell}^{N}(q)=q^{s(m,\ell,N)}\mathcal{C}_{m,\ell}^{N}(q) \label{equation:string-change}.
\end{equation}
We prove identities (\ref{equation:N4-string-evaluation-0}) and (\ref{equation:N4-string-evaluation-1}).   From \cite[Theorem 1.1]{MPS}, \cite[p. 219]{KP}, we have
\begin{equation*}
C_{0,0}^{4}(q)-C_{4,0}^{4}(q)=\frac{q^{-\tfrac{1}{12}}}{J_{1}^3}
\Big ( 
f_{3,3,1}(-q^{2},q,q)-qf_{3,3,1}(-q^4,q^3,q)
\Big )
= \frac{q^{-\tfrac{1}{12}}}{J_{1}^3} J_{1}J_{1,2},
\end{equation*}
where the last equality follows from (\ref{equation:N4-Hecke-evaluation-0}).   Similarly, we obtain an identity not in \cite{KP}:
\begin{equation*}
C_{0,0}^{4}(q)+C_{4,0}^{4}(q)
=\frac{q^{-\tfrac{1}{12}}}{J_{1}^3}
\Big ( 
f_{3,3,1}(q^{2},q,q)+qf_{3,3,1}(q^4,q^3,q)
\Big )
=q^{-\tfrac{1}{12}}\frac{1}{J_{1}^3} J_{1}\overline{J}_{3,6},
\end{equation*} 
where the last equality follows from (\ref{equation:N4-Hecke-evaluation-1}).    Hence
\begin{align*}
C_{0,0}^{4}(q)=\frac{q^{-\tfrac{1}{12}}}{2J_{1}^3}(J_{1}\overline{J}_{3,6}+J_{1}J_{1,2}), \\
C_{4,0}^{4}(q)=\frac{q^{-\tfrac{1}{12}}}{2J_{1}^3}(J_{1}\overline{J}_{3,6}-J_{1}J_{1,2}).
\end{align*}
The two identities (\ref{equation:N4-string-evaluation-0}) and (\ref{equation:N4-string-evaluation-1}) then follow from (\ref{equation:string-change}).

We prove identity (\ref{equation:N4-string-evaluation-2}).  From (\ref{equation:SW-fabc}), \cite[p. 219]{KP}, we have
\begin{equation*}
C_{2,0}^{4}(q)
=\frac{q^{-\tfrac{1}{3}}}{J_{1}^3}
f_{1,5,1}(q^{2},1,q)
=\frac{q^{\tfrac{2}{3}}}{J_{1}^3} J_{1}\overline{J}_{6,24},
\end{equation*} 
where the last equality follows from (\ref{equation:N4-Hecke-evaluation-4}).    Identity (\ref{equation:N4-string-evaluation-3}) then follows from (\ref{equation:string-change}).

We prove identities (\ref{equation:N4-string-evaluation-3}) and (\ref{equation:N4-string-evaluation-4}).  From \cite[Theorem 1.1]{MPS}, \cite[p. 220]{KP}, we have
\begin{equation*}
C_{1,1}^{4}(q^2)+C_{3,1}^{4}(q^2)=\frac{q^{-\tfrac{1}{6}}}{J_{2}^3}
\Big ( 
f_{3,3,1}(q^{5},q^4,q^2)+q f_{3,3,1}(q^{7},q^6,q^2)
\Big )
= \frac{q^{-\tfrac{1}{6}}}{J_{2}^3} J_{2}\overline{J}_{1,4},
\end{equation*}
where the last equality follows from (\ref{equation:N4-Hecke-evaluation-5}).   Similarly, we obtain an identity not in \cite{KP}:
\begin{equation*}
C_{1,1}^{4}(q^2)-C_{3,1}^{4}(q^2)=\frac{q^{-\tfrac{1}{6}}}{J_{2}^3}
\Big ( 
f_{3,3,1}(-q^{5},q^4,q^2)-q f_{3,3,1}(-q^{7},q^6,q^2)
\Big )
= \frac{q^{-\tfrac{1}{6}}}{J_{2}^3} J_{2}J_{1,4},
\end{equation*}
where the last equality follows from (\ref{equation:N4-Hecke-evaluation-6}).  Hence
\begin{align*}
C_{1,1}^{4}(q^2)&=\frac{q^{-\tfrac{1}{6}}}{2J_{2}^2}\cdot \Big ( \overline{J}_{1,4}+J_{1,4}\Big )
=\frac{q^{-\tfrac{1}{6}}}{J_{2}^2}\cdot  \overline{J}_{6,16},\\
C_{3,1}^{4}(q^2)&=\frac{q^{-\tfrac{1}{6}}}{2J_{2}^2}\cdot \Big ( \overline{J}_{1,4}-J_{1,4}\Big )
=\frac{q^{-\tfrac{1}{6}}}{J_{2}^2}\cdot  \overline{J}_{14,16},
\end{align*}
where we have used (\ref{equation:j-split}) with $m=2$.   Using (\ref{equation:1.7}), we have
\begin{align*}
C_{1,1}^{4}(q)&=\frac{q^{-\tfrac{1}{12}}}{J_{1}^2}\cdot  \overline{J}_{3,8},\\
C_{3,1}^{4}(q)&=\frac{q^{-\tfrac{1}{12}}}{J_{1}^2}\cdot  \overline{J}_{1,8},
\end{align*}
and the identities (\ref{equation:N4-string-evaluation-3}) and (\ref{equation:N4-string-evaluation-4}) follow from (\ref{equation:string-change}).

We prove identity (\ref{equation:N4-string-evaluation-5}).  From (\ref{equation:SW-fabc}), we obtain an identity which is not in \cite{KP}:
\begin{equation}
C_{0,2}^{4}(q)=\frac{q^{\tfrac{1}{4}}}{J_{1}^3}f_{1,5,1}(q^2,q^2,q)
=\frac{q^{\tfrac{1}{4}}}{J_{1}^3}J_{1}\overline{J}_{1,6},
\end{equation}
where the last equality follows from (\ref{equation:N4-Hecke-evaluation-2}).   Identity (\ref{equation:N4-string-evaluation-5}) then follows from (\ref{equation:string-change}).

We prove identity (\ref{equation:N4-string-evaluation-6}).   Using \cite[Corollary 1.3]{MPS}, we obtain an identity which is not in \cite{KP}:
\begin{equation}
C_{2,2}^{4}(q)=\frac{1}{J_{1}^3}f_{3,3,1}(q^3,q,q)
=\frac{1}{J_{1}^3}J_{1,4}J_{6,12},
\end{equation} 
where the last equality follows from (\ref{equation:N4-Hecke-evaluation-3}).  Identity (\ref{equation:N4-string-evaluation-6}) then follows from (\ref{equation:string-change}).  \qedhere

\end{proof}

\begin{proof}[Proof of Proposition \ref{proposition:N4-Hecke-evaluations}]  Identity (\ref{equation:N4-Hecke-evaluation-0}) is true by \cite[Lemma 3.11]{Mort2021}.

We prove (\ref{equation:N4-Hecke-evaluation-1}).  We recall Corollary \ref{corollary:f331-expansion}.   The contribution from (\ref{equation:g331-id}) reads
{\allowdisplaybreaks \begin{align*}
h_{3,3,1}&(q^2,q,q,-1,-1)+qh_{3,3,1}(q^4,q^3,q,-1,-1)\\
&=j(q^2;q^3)m(-q,q^2,-1)+j(q;q)m(-q^2,q^6,-1)\\
& \ \ \ \ \ +q\Big ( j(q^4;q^3)m(-q,q^2,-1)+j(q;q)m(-q^{-2},q^6,-1)\Big ) \\
&=j(q^2;q^3)m(-q,q^2,-1)+q j(q^4;q^3)m(-q,q^2,-1)\\
&=0,
\end{align*}}%
where we have used (\ref{equation:j-elliptic}).  Hence
{\allowdisplaybreaks \begin{align*}
f_{3,3,1}&(q^2,q,q)+qf_{3,3,1}(q^4,q^3,q)\\
&= -\sum_{d=0}^{2}
\frac{q^{d(d+1)}j\big (q^{3+2d};q^{3}\big )  j\big (-q^{5-2d};q^{6}\big ) J_{6}^3j\big (q^{3+2d};q^{6}\big )}
{4\overline{J}_{2,8}\overline{J}_{6,24}j\big (q^{2};q^6\big )j\big (q^{1+2d};q^6\big )}\\
& \ \ \ \ \ -q\sum_{d=0}^{2}
\frac{q^{d(d+1)}j\big (q^{5+2d};q^{3}\big )  j\big (-q^{5-2d};q^{6}\big ) J_{6}^3j\big (q^{-1+2d};q^{6}\big )}
{4\overline{J}_{2,8}\overline{J}_{6,24}j\big (q^{-2};q^6\big )j\big (q^{1+2d};q^6\big )}\\
&= \frac{J_{1}J_{6}^3}{2\overline{J}_{2,8}\overline{J}_{6,24}J_{2}J_{3,6}}\Big ( 
  \overline{J}_{3,6}J_{1,6}
 + \overline{J}_{1,6}J_{3,6} \Big ) \\
&= \frac{J_{1}J_{6}^3}{2\overline{J}_{2,8}\overline{J}_{6,24}J_{2}J_{3,6}} 
2j(q^4;q^{12})j(q^4;q^{12})\\
&= J_{1}\overline{J}_{3,6},
\end{align*}}%
where for the penultimate equality we used (\ref{equation:H1Thm1.2B}).

We prove (\ref{equation:N4-Hecke-evaluation-2}).  Using (\ref{equation:f-shift}) with $(R,S)=(-2,1)$, it is equivalent to show
\begin{equation}
-q^{9}f_{1,5,1}(q^{5},q^{-7},q)=J_{1}\overline{J}_{1,6}.
\end{equation}
We recall Theorem \ref{theorem:genfn4}.  We have
\begin{align*}
f_{1,5,1}(x,y,q)=g_{1,5,1}(x,y,q,y/x,x/y)-\Theta_{1,4}(x,y,q),
\end{align*}
where $x\rightarrow q^5$ and $y\rightarrow q^{-7}$ yield
\begin{equation*}
g_{1,5,1}(x,y,q,y/x,x/y)
\rightarrow j(q^{-7};q)m (-q^{54},q^{24},q^{-12} )
+ j(q^5;q)m(-q^{-18},q^{24},q^{12} )=0
\end{equation*}
and
{\allowdisplaybreaks \begin{align*}
\Theta_{1,4}(x,y,q)&\rightarrow \frac{q^{-1}}{j(-q^{30};q^{24})j(-q^{-18};q^{24})}\Big \{ J_{4,16}  S_1-qJ_{8,16} S_2\Big \}\\
&\ \ \ \ \ =\frac{q^{23}}{\overline{J}_{6,24}^2}\Big \{ J_{4,16}  S_1-qJ_{8,16} S_2\Big \},
\end{align*}}%
with
{\allowdisplaybreaks \begin{align*}
S_1&=\frac{j(q^{18};q^{24})j(-1;q^{24})j(q^{3};q^{12})}{J_{12}^3J_{48}} 
 \cdot j(-q^{6};q^{24})j(q^{-12};q^{24})J_{24}^2,\\
&=-q^{-12}\frac{J_{6,24}\overline{J}_{0,24}J_{3,12}}{J_{12}^3J_{48}} 
 \cdot \overline{J}_{6,24}J_{12,24}J_{24}^2,
\end{align*}}%
and
{\allowdisplaybreaks \begin{align*}
S_2&=\frac{j(q^{6};q^{24})j(-q^{-12};q^{24})j(q^{9};q^{12})}{J_{12}^2} 
 \cdot  \frac{q^{2}j(-q^{6},q^{-12};q^{24})J_{48}}{q^{-7}J_{24} }\\
 &=-q^{-15}\frac{J_{6,24}\overline{J}_{12,24}J_{3,12}}{J_{12}^2} 
 \cdot  \frac{\overline{J}_{6,24}J_{12,24}J_{48}}{J_{24} }.
\end{align*}}%
Hence
{\allowdisplaybreaks \begin{align*}
f_{1,5,1}(q^{5},q^{-7},q)
&=\frac{q^{9}}{\overline{J}_{6,24}^2}\Big [ q^2J_{4,16} 
\frac{J_{6,24}\overline{J}_{0,24}J_{3,12}}{J_{12}^3J_{48}} 
 \cdot \overline{J}_{6,24}J_{12,24}J_{24}^2\\
& \ \ \ \ \ -J_{8,16} \frac{J_{6,24}\overline{J}_{12,24}J_{3,12}}{J_{12}^2} 
 \cdot  \frac{\overline{J}_{6,24}J_{12,24}J_{48}}{J_{24} }\Big ]\\
&=\frac{q^{9}J_{3,12}J_{6,24}J_{12,24}}{\overline{J}_{6,24}}\Big [ q^2J_{4,16} 
\frac{\overline{J}_{0,24}}{J_{12}^3J_{48}} 
 \cdot J_{24}^2
  -J_{8,16} \frac{\overline{J}_{12,24}}{J_{12}^2} 
 \cdot  \frac{J_{48}}{J_{24} }\Big ]\\
&=\frac{q^{9}J_{3,12}J_{6,24}J_{12,24}}{\overline{J}_{6,24}J_{12}^3}\Big [ q^2\frac{J_{4}J_{16}J_{24}^2 }{J_{8}J_{48}} 
 \cdot \overline{J}_{0,24}
  -\frac{J_{8}^2 J_{12}J_{48}}{J_{16}J_{24} }\cdot \overline{J}_{12,24}\Big ]\\
&=\frac{q^{9}J_{3,12}J_{6,24}J_{12,24}J_{4}}{\overline{J}_{6,24}J_{12}^3}\Big [ q^2\overline{J}_{8,24}
 \cdot \overline{J}_{0,24}
  -\overline{J}_{4,24}\cdot \overline{J}_{12,24}\Big ]\\
&=-\frac{q^{9}J_{3,12}J_{6,24}J_{12,24}J_{4}}{\overline{J}_{6,24}J_{12}^3}\Big [ 
 \overline{J}_{4,24}\cdot \overline{J}_{12,24}
 -q^2\overline{J}_{8,24}
 \cdot \overline{J}_{0,24}\Big ].
\end{align*}}%
Using (\ref{equation:H1Thm1.1}) with $q\rightarrow q^{12}$, $x=y=q^2$ we have
\begin{equation*}
f_{1,5,1}(q^{5},q^{-7},q)
=-\frac{q^{9}J_{3,12}J_{6,24}J_{12,24}J_{4}}{\overline{J}_{6,24}J_{12}^3}
J_{2,12}^2
=-q^{9}J_{1}\overline{J}_{1,6}.
\end{equation*}

We prove (\ref{equation:N4-Hecke-evaluation-3}).    Here we show
\begin{equation}
\lim_{x\rightarrow q}f_{3,3,1}(x^3,x,q)=J_{1,4}J_{6,12}.
\end{equation}
We recall Corollary \ref{corollary:f331-expansion}.   The contribution from (\ref{equation:g331-id}) reads
{\allowdisplaybreaks \begin{align*}
\lim_{x\rightarrow q}&\Big ( h_{3,3,1}(x^3,x,q,-1,-1)\Big ) \\
&=\lim_{x\rightarrow q} \Big [ j(x^3;q^3)m(-q^2x^{-2};q^2,-1)+j(x;q)m(-q^3,q^6,-1)\Big ] \\
&=\lim_{x\rightarrow q} j(x^3;q^3)\Big [ m(-q^2x^{-2};q^2,z)
+z\frac{J_{2}^3j(-z^{-1};q^2)j(zq^2/x^2;q^2)}{\overline{J}_{0,2}j(z;q^2)j(q^2/x^2;q^2)j(-zq^2/x^2;q^2)}\Big ] \\
&=\lim_{x\rightarrow q} j(x^3;q^3)\Big [ m(-q^2x^{-2};q^2,q)
+\frac{J_{2}^3\overline{J}_{1,2}j(q^3/x^2;q^2)}{\overline{J}_{0,2}J_{1,2}j(q^2/x^2;q^2)j(-q^3/x^2;q^2)}\Big ] \\
&=-\lim_{x\rightarrow q} j(x^3;q^3)\frac{J_{2}^3\overline{J}_{1,2}j(qx^2;q^2)}{\overline{J}_{0,2}J_{1,2}j(x^2;q^2)j(-qx^2;q^2)}\\
&=-\lim_{x\rightarrow q} j(x;q)j(x\omega;q)j(x\omega^2;q)\frac{J_{3}}{J_{1}^3}
\frac{J_{2}^3\overline{J}_{1,2}j(qx^2;q^2)}{\overline{J}_{0,2}J_{1,2}j(x;q)j(-x;q)j(-qx^2;q^2)}
\frac{J_{1}^2}{J_{2}}\\
&=j(q\omega;q)j(q\omega^2;q)\frac{J_{3}}{J_{1}^3}
\frac{J_{2}^3}{\overline{J}_{0,2}\overline{J}_{0,1}}
\frac{J_{1}^2}{J_{2}}\\
&=\frac{3}{2} \frac{J_{3}^3}{\overline{J}_{0,2}},
\end{align*}}%
where the second equality follows from (\ref{equation:changing-z-theorem}).  Hence
{\allowdisplaybreaks \begin{align*}
\lim_{x\rightarrow q}&f_{3,3,1}(x^3,x,q)\\
&=\frac{3}{2} \frac{J_{3}^3}{\overline{J}_{0,2}}-\lim_{x\rightarrow q}\sum_{d=0}^{2}
\frac{q^{d(d+1)}J_{6}^3 j\big (q^{2+2d}x;q^{3}\big )  j\big (-q^{4-2d}x^2;q^{6}\big ) j\big (q^{5+2d}x^{-2};q^{6}\big )}
{4\overline{J}_{2,8}\overline{J}_{6,24}j\big (q^{3};q^6\big )j\big (q^{2+2d}x^{-2};q^6\big )}\\
&=\frac{3}{2} \frac{J_{3}^3}{\overline{J}_{0,2}} - \lim_{x\rightarrow q}\Big [\frac{J_{6}^3 j\big (q^{2}x;q^{3}\big )  j\big (-q^{4}x^2;q^{6}\big ) j\big (q^{5}x^{-2};q^{6}\big )}
{4\overline{J}_{2,8}\overline{J}_{6,24}J_{3,6}j\big (q^{2}x^{-2};q^6\big )}\\
& \ \ \ \ \ + \frac{q^{2}J_{6}^3 j\big (q^{4}x;q^{3}\big )  j\big (-q^{2}x^2;q^{6}\big ) j\big (q^{7}x^{-2};q^{6}\big )}
{4\overline{J}_{2,8}\overline{J}_{6,24}J_{3,6}j\big (q^{4}x^{-2};q^6\big )}\\
& \ \ \ \ \ + \frac{q^{6}J_{6}^3 j\big (q^{6}x;q^{3}\big )  j\big (-x^2;q^{6}\big ) j\big (q^{9}x^{-2};q^{6}\big )}
{4\overline{J}_{2,8}\overline{J}_{6,24}J_{3,6}j\big (q^{6}x^{-2};q^6\big )}\Big ] \\
&= \frac{3}{2} \frac{J_{3}^3}{\overline{J}_{0,2}}- \lim_{x\rightarrow q}\Big [\frac{J_{6}^3 j\big (q^{2}x;q^{3}\big )  j\big (-q^{4}x^2;q^{6}\big ) 
j\big (qx^{2};q^{6}\big )}
{\overline{J}_{0,2}\overline{J}_{0,6}J_{3,6}j\big (q^{4}x^{2};q^6\big )}\\
& \ \ \ \ \ + \frac{xJ_{6}^3 j\big (qx;q^{3}\big )  j\big (-q^{2}x^2;q^{6}\big ) j\big (qx^{-2};q^{6}\big )}
{\overline{J}_{0,2}\overline{J}_{0,6}J_{3,6}j\big (q^{4}x^{-2};q^6\big )}\\
& \ \ \ \ \ - \frac{J_{6}^3 j\big (x;q^{3}\big )  j\big (-x^2;q^{6}\big ) j\big (q^{3}x^{-2};q^{6}\big )}
{\overline{J}_{0,2}\overline{J}_{0,6}J_{3,6}j\big (x^{2};q^6\big )}\Big ] \\
&=\frac{3}{2} \frac{J_{3}^3}{\overline{J}_{0,2}} - \lim_{x\rightarrow q}\Big [\frac{J_{6}^3 j\big (q^{2}x;q^{3}\big )  j\big (-q^{4}x^2;q^{6}\big ) 
j\big (qx^{2};q^{6}\big )}
{\overline{J}_{0,2}\overline{J}_{0,6}J_{3,6}j\big (q^{2}x;q^3\big )j\big (-q^{2}x;q^3\big )}\frac{J_{3}^2}{J_{6}}\Big ] 
 + 2\frac{J_{6}^3 J_{1}\overline{J}_{2,6} J_{1,6}}
{\overline{J}_{0,2}\overline{J}_{0,6}J_{3,6}J_{2}}\\
&= \frac{3}{2} \frac{J_{3}^3}{\overline{J}_{0,2}}- \frac{1}{2}\frac{J_{3}^3} 
{\overline{J}_{0,2}}
 + 2\frac{J_{6}^3 J_{1}\overline{J}_{2,6} J_{1,6}}
{\overline{J}_{0,2}\overline{J}_{0,6}J_{3,6}J_{2}}.
\end{align*}}%
Continuing, we have
{\allowdisplaybreaks \begin{align*}
f_{3,3,1}(q^3,q,q)&=\frac{3}{2} \frac{J_{3}^3}{\overline{J}_{0,2}} - \frac{1}{2}\frac{J_{3}^3} 
{\overline{J}_{0,2}}
 + 2\frac{J_{6}^3 J_{1}\overline{J}_{2,6} J_{1,6}}
{\overline{J}_{0,2}\overline{J}_{0,6}J_{3,6}J_{2}}\\
&= \frac{J_{3}^3}{\overline{J}_{0,2}}\cdot  \Big [ 1
 + \frac{\overline{J}_{3,6}\overline{J}_{3,12}}{\overline{J}_{1,6}\overline{J}_{1,3}}\Big ]\\
 &= \frac{1}{2}\cdot \frac{J_{1}^2J_{6}^2}{J_{2}^2J_{4}J_{12}} \Big [\overline{J}_{1,6}\overline{J}_{1,3}
 + \overline{J}_{3,6}\overline{J}_{3,12}\Big ]\\
  &= J_{1,4}J_{6,12}\cdot \frac{2}{\overline{J}_{0,1}\overline{J}_{0,2}} \Big [\overline{J}_{1,6}\overline{J}_{1,3}
 + \overline{J}_{3,6}\overline{J}_{3,12}\Big ]\\
 &=J_{1,4}J_{6,12},
\end{align*}}%
where the last equality follows from Lemma \ref{lemma:N4-theta-evaluation}.

We prove (\ref{equation:N4-Hecke-evaluation-4}).   Using (\ref{equation:f-shift}) with $(R,S)=(0,1)$, it is equivalent to show
\begin{align*}
f_{1,5,1}(q^7,q,q)=-qJ_{1}\overline{J}_{6,24}.
\end{align*}
We recall Theorem \ref{theorem:genfn4}.   Arguing as in the proof of identity (\ref{equation:N4-Hecke-evaluation-2}), we find that under the substitutions $x\rightarrow q^7$ and $y\rightarrow q$, we have
\begin{equation*}
g_{1,5,1}(x,y,q,y/x,x/y)
\rightarrow j(q;q)m\Big (-q^{16},q^{24},q^{-6}\Big )
 + j(q^7;q)m\Big (-q^{-20},q^{24},q^6\Big )=0
\end{equation*}
and that 
{\allowdisplaybreaks \begin{align*}
\Theta_{1,4}(x,y,q)
&\rightarrow \frac{q^9}{j(-q^{38};q^{24})j(-q^{14};q^{24})}\Big \{ J_{4,16}  S_1-qJ_{8,16} S_2\Big \}\\
&\ \ \ \ \ =\frac{q^{23}}{\overline{J}_{10,24}^2}\Big \{ J_{4,16}  S_1-qJ_{8,16} S_2\Big \}
\end{align*}}%
with
\begin{align*}
S_1&=\frac{j(q^{38};q^{24})j(-q^{6};q^{24})j(q^{13};q^{12})}{J_{12}^3J_{48}}
 \cdot \frac{q^{19} j(-q^{38};q^{24})j(q^{6};q^{24})^2j(-q^{-6};q^{24})^2}{J_{24}}\\
 &=q^{-22}\frac{J_{10,24}\overline{J}_{6,24}J_{1,12}}{J_{12}^3J_{48}}
 \cdot \frac{\overline{J}_{10,24}J_{6,24}^2\overline{J}_{6,24}^2}{J_{24}}
 \end{align*}
 and
{\allowdisplaybreaks \begin{align*}
S_2&=\frac{j(q^{26};q^{24})j(-q^{-6};q^{24})j(q^{19};q^{12})}{J_{12}^2} 
 \cdot \frac{q^{8}j(-q^{38};q^{24})j(q^{12};q^{48})^2}{J_{48} }\\
 &=q^{-21}\frac{J_{2,24}\overline{J}_{6,24}J_{5,12}}{J_{12}^2} 
 \cdot \frac{\overline{J}_{10,24}J_{12,48}^2}{J_{48} }.
\end{align*}}%
Assembling the pieces, we have
{\allowdisplaybreaks \begin{align*}
f_{1,5,1}(q^7,q,q)
&=-\frac{q^{23}}{\overline{J}_{10,24}^2}\Big [ J_{4,16} q^{-22}\frac{J_{10,24}\overline{J}_{6,24}J_{1,12}}{J_{12}^3J_{48}}
 \cdot \frac{\overline{J}_{10,24}J_{6,24}^2\overline{J}_{6,24}^2}{J_{24}}\\ 
& \ \ \ \ \ -q^{-20}J_{8,16} \frac{J_{2,24}\overline{J}_{6,24}J_{5,12}}{J_{12}^2} 
 \cdot \frac{\overline{J}_{10,24}J_{12,48}^2}{J_{48} }\Big ]\\
 &=-\frac{q\overline{J}_{6,24}J_{12,48}^2}{\overline{J}_{10,24}}\Big [ J_{4,16} \frac{J_{10,24}J_{1,12}}{J_{12}^3J_{48}}
 \cdot \frac{J_{24}^3}{J_{48}^2}
  -q^2J_{8,16} \frac{J_{2,24}J_{5,12}}{J_{12}^2} 
 \cdot \frac{1}{J_{48} }\Big ]\\
 &=-\frac{q\overline{J}_{6,24}J_{12,48}^2J_{1,12}J_{5,12}}{\overline{J}_{10,24}}\frac{J_{24}}{J_{12}^2}
 \Big [ J_{4,16} \frac{\overline{J}_{5,12}}{J_{12}^3J_{48}}
 \cdot \frac{J_{24}^3}{J_{48}^2}
  -q^2J_{8,16} \frac{\overline{J}_{1,12}}{J_{12}^2J_{48} }\Big ]\\
 &=-qJ_{1}\overline{J}_{6,24}\cdot 
 \frac{J_{4}}{J_{2}J_{6}J_{12}}
 \frac{1}{\overline{J}_{10,24}}
\cdot  \Big [ \overline{J}_{8,24} \overline{J}_{5,12}  \overline{J}_{3,12}
  -q^2\overline{J}_{4,24}  \overline{J}_{1,12} \overline{J}_{3,12}\Big ]\\
 &=-qJ_{1}\overline{J}_{6,24}\cdot 
 \frac{J_{4}}{J_{2}J_{6}J_{12}}
 \frac{1}{\overline{J}_{10,24}}\frac{J_{24}}{J_{12}^2}\\
& \ \ \ \ \ \cdot  \Big [ j(iq^4;q^{12}) j(-iq^4;q^{12}) \overline{J}_{5,12}  \overline{J}_{3,12}
  -q^2j(iq^2;q^{12}) j(-iq^2;q^{12})  \overline{J}_{1,12} \overline{J}_{3,12}\Big ],
 \end{align*}}%
where for the last equality we used (\ref{equation:1.12}).  Using Proposition \ref{proposition:Weierstrass} with $q\rightarrow q^{12}$, $a=-q^5$, $b=q^4$, $c=q^2$, $d=-i$ yields
 {\allowdisplaybreaks \begin{align*}
f_{1,5,1}(q^7,q,q)
 =-q\frac{J_{1}\overline{J}_{6,24}J_{4}J_{24}}{J_{2}J_{6}J_{12}\overline{J}_{10,24}J_{12}^2}
  \cdot  j(iq^5;q^{12}) j(-iq^5;q^{12}) J_{6,12}  J_{2,12}
  =-qJ_{1}\overline{J}_{6,24}.
 \end{align*}}%

We prove (\ref{equation:N4-Hecke-evaluation-5}).   We recall Corollary \ref{corollary:f331-expansion}.   The contribution from (\ref{equation:g331-id}) reads
{\allowdisplaybreaks \begin{align*}
h_{3,3,1}&( q^{5},q^4,q^2,-1,-1)+ qh_{3,3,1}( q^{7},q^6,q^2,-1,-1)\\
&=j(q^5;q^6)m(-q^3,q^4,-1)+j(q^4;q^2)m(-q^{-1},q^{12},-1)\\
& \ \ \ \ \ q\Big (j(q^7;q^6)m(-q^3,q^4,-1)+j(q^6;q^2)m(-q,q^{12},-1)\Big )\\
&=j(q^5;q^6)m(-q^3,q^4,-1)+ q j(q^7;q^6)m(-q^3,q^4,-1)\\
&=0,
\end{align*}}%
where for the last equality we used (\ref{equation:j-elliptic}).  Thus
{\allowdisplaybreaks \begin{align*}
f_{3,3,1}&( q^{5},q^4,q^2)+ qf_{3,3,1}( q^{7},q^6,q^2)\\
&=-\sum_{d=0}^{2}\frac{q^{2d(d+1)}j\big (q^{8+4d};q^{6}\big )  j\big (-q^{9-4d};q^{12}\big ) 
J_{12}^3j\big (q^{2+4d};q^{12}\big )}
{\overline{J}_{0,4}\overline{J}_{0,12}j\big (q^{-1};q^{12}\big )j\big (q^{3+4d};q^{12}\big )}\\
& \ \ \ \ \ -q\sum_{d=0}^{2}\frac{q^{2d(d+1)}j\big (q^{10+4d};q^{6}\big )  j\big (-q^{9-4d};q^{12}\big ) 
J_{12}^3j\big (q^{-2+4d};q^{12}\big )}
{\overline{J}_{0,4}\overline{J}_{0,12}j\big (q^{-5};q^{12}\big )j\big (q^{3+4d};q^{12}\big )}\\
&=-\frac{j\big (q^{8};q^{6}\big )  j\big (-q^{9};q^{12}\big ) 
J_{12}^3j\big (q^{2};q^{12}\big )}
{\overline{J}_{0,4}\overline{J}_{0,12}j\big (q^{-1};q^{12}\big )j\big (q^{3};q^{12}\big )}
 -\frac{q^{12}j\big (q^{16};q^{6}\big )  j\big (-q;q^{12}\big ) 
J_{12}^3j\big (q^{10};q^{12}\big )}
{\overline{J}_{0,4}\overline{J}_{0,12}j\big (q^{-1};q^{12}\big )j\big (q^{11};q^{12}\big )}\\
& \ \ \ \ \ -q\frac{j\big (q^{10};q^{6}\big )  j\big (-q^{9};q^{12}\big ) 
J_{12}^3j\big (q^{-2};q^{12}\big )}
{\overline{J}_{0,4}\overline{J}_{0,12}j\big (q^{-5};q^{12}\big )j\big (q^{3};q^{12}\big )}
 -q\frac{q^{4}j\big (q^{14};q^{6}\big )  j\big (-q^{5};q^{12}\big ) 
J_{12}^3j\big (q^{2};q^{12}\big )}
{\overline{J}_{0,4}\overline{J}_{0,12}j\big (q^{-5};q^{12}\big )j\big (q^{7};q^{12}\big )}\\
&=-q^{-1}\frac{J_{2}  \overline{J}_{3,12} J_{12}^3J_{2,12}}
{\overline{J}_{0,4}\overline{J}_{0,12}J_{1,12}J_{3,12}}
 +q^{-1}\frac{J_{2}  \overline{J}_{1,12} J_{12}^3J_{2,12}}
{\overline{J}_{0,4}\overline{J}_{0,12}J_{1,12}^2}\\
&\ \ \ \ \  +\frac{J_{2}  \overline{J}_{3,12} J_{12}^3J_{2,12}}
{\overline{J}_{0,4}\overline{J}_{0,12}J_{5,12}J_{3,12}}
 +\frac{J_{2} \overline{J}_{5,12} J_{12}^3J_{2,12}}
{\overline{J}_{0,4}\overline{J}_{0,12}J_{5,12}^2},
\end{align*}}%
where we have simplified using (\ref{equation:j-elliptic}).  Regrouping terms, we have
{\allowdisplaybreaks \begin{align*}
f_{3,3,1}&( q^{5},q^4,q^2)+ qf_{3,3,1}( q^{7},q^6,q^2)\\
&=-q^{-1}\frac{J_{2} J_{12}^3J_{2,12}}
{\overline{J}_{0,4}\overline{J}_{0,12}J_{1,12}}
\cdot \Big ( \frac{\overline{J}_{3,12}}{J_{3,12}}- \frac{\overline{J}_{1,12}}{J_{1,12}}\Big ) 
+\frac{J_{2} J_{12}^3J_{2,12}}
{\overline{J}_{0,4}\overline{J}_{0,12}J_{5,12}} 
\cdot \Big ( \frac{\overline{J}_{3,12}}{J_{3,12}}+ \frac{\overline{J}_{5,12}}{J_{5,12}}\Big ) \\
&=-q^{-1}\frac{J_{2} J_{12}^3J_{2,12}}
{\overline{J}_{0,4}\overline{J}_{0,12}J_{1,12}}
\cdot \Big ( \frac{-2qJ_{2,24}J_{16,24}}{J_{3,12}J_{1,12}}\Big ) 
+\frac{J_{2} J_{12}^3J_{2,12}}
{\overline{J}_{0,4}\overline{J}_{0,12}J_{5,12}} 
\cdot \Big (  \frac{2J_{8,24}J_{10,24}}{J_{3,12}J_{5,12}}\Big ) \\
&=2\frac{J_{2} J_{12}^3J_{2,12}J_{8}}
{\overline{J}_{0,4}\overline{J}_{0,12}J_{3,12}}
\cdot \Big ( \frac{J_{2,24}}{J_{1,12}^2}+\frac{J_{10,24}}{J_{5,12}^2}\Big ) \\
&=2\frac{J_{2} J_{12}^3J_{2,12}J_{8}}
{\overline{J}_{0,4}\overline{J}_{0,12}J_{3,12}}\cdot\frac{J_{24}}{J_{12}^2}
\cdot \Big ( \frac{\overline{J}_{1,12}}{J_{1,12}}+\frac{\overline{J}_{5,12}}{J_{5,12}}\Big ) \\
&=2\frac{J_{2} J_{12}^3J_{2,12}J_{8}}
{\overline{J}_{0,4}\overline{J}_{0,12}J_{3,12}}\cdot\frac{J_{24}}{J_{12}^2}
\cdot \frac{2J_{6,24}J_{16,24}}{J_{1,12}J_{5,12}} \\
&=J_{2}\overline{J}_{1,4},
\end{align*}}%
where we used (\ref{equation:H1Thm1.2A}) and  (\ref{equation:H1Thm1.2B}) for the second equality, regrouped terms, used elementary product rearrangements, used (\ref{equation:H1Thm1.2B}) for the penultimate equality, and then finished with more product rearrangements.

We prove (\ref{equation:N4-Hecke-evaluation-6}).   This follows from substituting $q\rightarrow -q$ in  (\ref{equation:N4-Hecke-evaluation-5}).  \qedhere

\end{proof}


\section{Computing level $N=2$ string functions: Examples}  \label{section:LevelN2-KP} 

\subsection{The string function $c_{20}^{20}-c_{02}^{20}$ (\ref{equation:KP-2-List2-A}):}
We give two proofs of identity (\ref{equation:KP-2-List2-A}).

For the first proof, we use Theorem \ref{theorem:Level3-theorem} to obain
{\allowdisplaybreaks \begin{align}
c_{20}^{20}=C_{0,0}^{2}(q)=\frac{q^{-\frac{1}{16}}}{J_{1}^3}\cdot J_{1,2}\overline{J}_{3,8},\\
c_{02}^{20}=C_{0,0}^{2}(q)=\frac{q^{\frac{7}{16}}}{J_{1}^3}\cdot J_{1,2}\overline{J}_{1,8}.
\end{align}}%
Combining terms and using (\ref{equation:1.7}), we have
\begin{equation*}
c_{20}^{20}-c_{02}^{20}
=\frac{q^{-\frac{1}{16}}}{J_{1}J_{2}}\cdot \Big ( \overline{J}_{3,8}-q^{\frac{1}{2}}\overline{J}_{7,8}\Big ) 
=\frac{q^{-\frac{1}{16}}}{J_{1}J_{2}}\cdot j(q^{1/2};q^2)
=\frac{q^{-\frac{1}{16}}}{J_{1}^2}j(q^{1/2};q^{3/2}),
\end{equation*}
where used (\ref{equation:j-split}) with $m=2$ and the two product rearrangements $J_{1,2}=J_{1}^2/J_{2}$ and $J_{1,4}=J_{1}J_{4}/J_{2}$.

For the second proof, we use \cite[Theorem $1.1$]{MPS} to obtain
\begin{align*}
c_{20}^{20}(q^2)-c_{02}^{20}(q^2)
&=C_{0,0}^{2}(q^2)-C_{2,0}^2(q^2)\\
&=\frac{q^{-\frac{1}{8}}}{J_{2}^3}\cdot \Big ( f_{2,2,1}(-q^3,q^2,q^2)-qf_{2,2,1}(-q^5,q^4,q^2)\Big ). 
\end{align*}
We next recall Corollary \ref{corollary:f221-expansion}.   We have
\begin{align*}
h_{2,2,1}&(-q^3,q^2,q^2,-1,-1)-qh_{2,2,1}(-q^5,q^4,q^2,-1,-1)\\
&=j(-q^3;q^4)m(q,q^2,-1)+j(q^2;q^2)m(-q,q^4,-1)\\
& \ \ \ \ \ -q\cdot \Big ( j(-q^5;q^4)m(q,q^2,-1)+j(q^4;q^2)m(-q,q^4,-1)\Big ) \\
&=j(-q^3;q^4)m(q,q^2,-1)-q j(-q^5;q^4)m(q,q^2,-1)\\
&=0,
\end{align*} 
where the last equality follows from (\ref{equation:j-elliptic}).  Hence from Corollary \ref{corollary:f221-expansion}:
{\allowdisplaybreaks \begin{align*}
 f_{2,2,1}&(-q^3,q^2,q^2)-qf_{2,2,1}(-q^5,q^4,q^2)\\
&=-\frac{q^2j\big (q^{6};q^{4}\big )  j\big (q;q^{4}\big ) 
J_{4}^3j\big (-q^{4};q^{4}\big )}
{4\overline{J}_{2,8}\overline{J}_{4,16}j\big (q;q^4\big )j\big (-q^{3};q^4\big )}
 +q\frac{j\big (q^6;q^{4}\big )  j\big (-q^3;q^{4}\big ) 
J_{4}^3j\big (-1;q^{4}\big )}
{4\overline{J}_{2,8}\overline{J}_{4,16}j\big (q^{-1};q^4\big )j\big (-q;q^4\big )} \\
&=\frac{J_{2,4}  J_{1,4} 
J_{4}^3\overline{J}_{0,4}}
{2\overline{J}_{2,8}\overline{J}_{4,16}J_{1,4}\overline{J}_{1,4}}\\
&=J_{1}J_{2},
\end{align*}}%
where the last equality follows from product rearrangements.

\section{Computing level $N=3$ string functions: Examples}  \label{section:LevelN3-KP}

\subsection{The string function $c_{12}^{30}$ (\ref{equation:KP-3-List2-A}):}

Using Theorem \ref{theorem:Level3-theorem} gives
\begin{equation}
c_{12}^{30}=C_{2,0}^{3}(q)
=\frac{q^{\frac{71}{120}}}{J_{1}^2}\cdot J_{3,15}.
\end{equation}

\subsection{The string function $c_{30}^{30}-c_{12}^{30}$ (\ref{equation:KP-3-List2-B}):}

Using Theorem \ref{theorem:Level3-theorem} gives
\begin{equation}
c_{30}^{30}=C_{0,0}^{3}(q)
=\frac{q^{-\frac{3}{40}}}{J_{1}^2}\cdot (J_{8,15}-qJ_{2,15}).
\end{equation}
From (\ref{equation:j-split}) with $m=3$ and (\ref{equation:j-elliptic}), we have
\begin{equation*}
J_{2,5}=J_{21,45}-q^2J_{36,45}+q^3J_{6,45},
\end{equation*}
and under the substitution $q\rightarrow q^{1/3}$ we have
\begin{equation*}
j(q^{2/3};q^{5/3})=J_{7,15}-q^{2/3}J_{12,15}-qJ_{2,15}.
\end{equation*}
Identity (\ref{equation:KP-3-List2-B}) is now straightforward
\begin{equation*}
c_{30}^{30}-c_{12}^{30}
=\frac{q^{-\frac{3}{40}}}{J_{1}^2}\cdot (J_{8,15}
-qJ_{2,15}-q^{2/3}J_{3,15})
=\frac{q^{-\frac{3}{40}}}{J_{1}^2}\cdot j(q^{2/3};q^{5/3}).
\end{equation*}

\subsection{The string function $c_{21}^{21}-c_{03}^{21}$ (\ref{equation:KP-3-List2-C}):}

Using Theorem \ref{theorem:Level3-theorem} gives
\begin{align}
c_{21}^{21}=C_{1,1}^{3}(q)&=\frac{q^{-\frac{1}{120}}}{J_{1}^2}\cdot J_{6,15},\\
c_{03}^{21}=C_{3,1}^{3}(q)&=\frac{q^{\frac{13}{40}}}{J_{1}^2}\cdot (J_{11,15}+qJ_{1,15}).
\end{align}
From (\ref{equation:j-split}) with $m=3$ and (\ref{equation:j-elliptic}), we have
\begin{equation*}
J_{1,5}=J_{18,45}-qJ_{33,45}-q^4J_{3,45}.
\end{equation*}
The substitution $q\rightarrow q^{1/3}$ yields
\begin{equation*}
j(q^{1/3};q^{5/3})=J_{6,15}-q^{1/3}J_{11,15}-q^{4/3}J_{1,15}.
\end{equation*}
Hence 
\begin{equation*}
c_{21}^{21}-c_{03}^{21}
=\frac{q^{-\frac{1}{120}}}{J_{1}^2} \Big [ J_{6,15} -q^{1/3}\cdot (J_{11,15} +qJ_{1,15}) \Big ] 
=\frac{q^{-\frac{1}{120}}}{J_{1}^2}\cdot j(q^{1/3};q^{5/3}).
\end{equation*}


\section{Computing level $N=4$ string functions: Examples}  \label{section:LevelN4-KP} 
\subsection{The string function $c_{40}^{40}-2c_{22}^{40}+c_{04}^{40}+2c_{04}^{22}-2c_{22}^{22}$ (\ref{equation:KP-4-List2-B})}  Using Theorem \ref{theorem:Level3-theorem} gives
{\allowdisplaybreaks \begin{align}
c_{40}^{40}&=C_{0,0}^{4}(q)=\frac{q^{-\frac{1}{12}}}{2J_{1}^2}\cdot (\overline{J}_{3,6}+J_{1,2}),\\
c_{22}^{40}&=C_{2,0}^{4}(q)=\frac{q^{\frac{2}{3}}}{J_{1}^2}\overline{J}_{6,24},\\
c_{04}^{40}&=C_{4,0}^{4}(q)=\frac{q^{-\frac{1}{12}}}{2J_{1}^2}\cdot (\overline{J}_{3,6}-J_{1,2}),\\
c_{40}^{22}&=C_{0,2}^{4}(q)=\frac{q^{\frac{1}{4}}}{J_{1}^2}\overline{J}_{1,6},\\
c_{22}^{22}&=C_{2,2}^{4}(q)=\frac{1}{J_1^2}\overline{J}_{2,6}.
\end{align}}%
Hence
{\allowdisplaybreaks \begin{align*}
&c_{40}^{40}-2c_{22}^{40}+c_{04}^{40}+2c_{04}^{22}-2c_{22}^{22}\\
&\ \ \ \ \ =\frac{1}{J_{1}^2}\cdot \Big ( q^{-\frac{1}{12}} \overline{J}_{3,6}
-2q^{\frac{2}{3}}\overline{J}_{6,24}
+2q^{\frac{1}{4}}\overline{J}_{1,6}
-2\overline{J}_{2,6}\Big ) \\
&\ \ \ \ \ =\frac{1}{J_{1}^2}\cdot \Big ( q^{-\frac{1}{12}}(\overline{J}_{12,24}+q^{3}\overline{J}_{0,24}) 
-2q^{\frac{2}{3}}\overline{J}_{6,24}
+2q^{\frac{1}{4}}(\overline{J}_{8,24}+q\overline{J}_{20,24})\\
& \ \ \ \ \ \ \ \ \ \ -2(\overline{J}_{10,24}+q^2\overline{J}_{22,24})\Big ) \\
&\ \ \ \ \ =\frac{q^{-\frac{1}{12}}}{J_{1}^2}\cdot \Big ( \overline{J}_{12,24}+q^{3}\overline{J}_{0,24}
-2q^{\frac{3}{4}}\overline{J}_{6,24}
+2q^{\frac{1}{3}}\overline{J}_{8,24}+2q^{\frac{4}{3}}\overline{J}_{20,24}\\
&\ \ \ \ \ \ \ \ \ \ -2q^{\frac{1}{12}}\overline{J}_{10,24}-2q^{\frac{25}{12}}\overline{J}_{22,24}\Big )\\
&\ \ \ \ \ =\frac{q^{-\frac{1}{12}}}{J_{1}^2}\cdot j(q^{1/12};q^{1/6}),
\end{align*}}%
where the second equality follows from (\ref{equation:j-split}) with $m=2$, and the last equality follows from Lemma \ref{lemma:levelN4-jsplit}.


\section*{Acknowledgements}
We would like to thank O. Warnaar for helpful comments and suggestions.  This research was supported by Ministry of Science and Higher Education of the Russian Federation, agreement No. 075-15-2019-1619.

\end{document}